\documentclass[final]{siamltex}

\usepackage{graphicx} 
\usepackage{amsmath}
\usepackage{amsfonts}
\usepackage{float}

\usepackage{adjustbox}

\usepackage[usenames,dvipsnames]{color}
\newcommand{\MG}[1]{{\color{black} #1}}
\renewcommand{\vec}{\mathbf}
\newcommand{\VD}[1]{{\color{black} #1}}

\DeclareMathOperator{\expo}{e}

\DeclareMathOperator{\Part}{Part}
\DeclareMathOperator{\csgn}{csgn}
\DeclareMathOperator{\sg}{sg}

\title{Natural domain decomposition algori\MG{thms} for the solution of
  time-harmonic elastic waves}

\author{R. Brunet\thanks{Department of Mathematics and Statistics, University of Strathclyde, Glasgow, UK, E-mail: {Romain.Brunet@strath.ac.uk}.}
\and V. Dolean\thanks{Department of Mathematics and Statistics, University of Strathclyde, Glasgow, UK, and Laboratoire J.A.~Dieudonn\'e, CNRS, University C\^ote d'Azur, Nice, France. E-mail: {work@victoritadolean.com}.}   
\and M. J. Gander \thanks{Universit\'e de Gen\`eve, 2-4 rue du
  Li\`evre, Gen\`eve. E-mail: {martin.gander@unige.ch}.}
}

\begin{document}

\maketitle

\begin{abstract}
We study for the first time Schwarz domain decomposition methods for
the solution of the Navier equations modeling the propagation of
elastic waves. These equations in the time harmonic regime are
difficult to solve by iterative methods, even more so than the
Helmholtz equation. We first prove that the classical Schwarz method
is not convergent when applied to the Navier equations, and can thus
not be used as an iterative solver, only as a preconditioner for a
Krylov method.  We then introduce more natural transmission conditions
between the subdomains, and show that if the overlap is not too small,
this new Schwarz method is convergent. We illustrate our results with
numerical experiments, both for situations covered by our technical
two subdomain analysis, and situations that go far beyond, including
many subdomains, cross points, heterogeneous materials \MG{in a
  transmission problem}, and Krylov acceleration. Our numerical
results show that the Schwarz method with adapted transmission
conditions leads systematically to a better solver for the Navier
equations than the classical Schwarz method.
\end{abstract}

\begin{keywords} 
Domain decomposition methods, Schwarz preconditioners, time-harmonic
elastic waves, Navier equations.
\end{keywords}

\begin{AMS}
65N55, 65N35, 65F10
\end{AMS}

\pagestyle{myheadings}
\thispagestyle{plain}
\markboth{R. Brunet, V. Dolean, M.J. Gander}{Solvers and preconditioners for time-harmonic elastic waves}

\section{Introduction}

Time harmonic problems are difficult to solve by iterative methods in
the medium to high frequency regime, see \cite{Ernst:12:NAM} for the
case of the Helmholtz equation, which is the prototype of such time
harmonic problems with oscillatory solutions. Domain decomposition
methods are a natural choice as iterative solvers for such problems,
since they are by construction parallel and can still locally use
direct solvers without convergence problems. To obtain good domain
decomposition convergence for time harmonic problems, adapted
transmission conditions are however needed between subdomains.  Such
transmission conditions were first studied for the Helmholtz equation
by Despr\`es in \cite{Despres:1990:DDP,Despres:1991:MDD}, and later
optimized variants were introduced and analyzed by Chevalier in his
PhD thesis \cite{Chevalier:1998:MNT}, see also \MG{Chevalier and
  Nataf} \cite{Chevalier:1998:SMO}, \MG{the work by Collino, Delbue,
  Joly and Piacentini} \cite{Collino:1997:NIM}, and Gander et
al. \cite{Gander:2001:OSH,gander2007optimized,gander2016optimized}. Very
similar in nature to the Helmholtz equations, high-frequency
time-harmonic Maxwell's equations are also very difficult to solve
iteratively, and the design of efficient domain decomposition methods
for the intermediate to high frequency regime is even harder. First
optimized transmission conditions both for the first and second order
formulations of Maxwell's equations can already be found in the PhD
thesis of Chevalier \cite[section 4.7]{Chevalier:1998:MNT} and Collino
\MG{et al.} \cite{Collino:1997:NIM}, but were then more systematically
developed by Alonso\MG{-Rodriguez and Gerardo-Giorda}
\cite{Alonso:06:NND}, and especially in Dolean et
  al. \cite{Dolean:08:DDM,Dolean:09:OSM,ElBouajaji:12:OSM}, see also
  Peng, Rawat and Lee \cite{Peng:10:ODD}\MG{, and references therein.}  The
Analytic Incomplete LU factorization (AILU) \cite{GanderAILU05}, the
sweeping preconditioner \cite{EY1,EY2}, the source transfer domain
decomposition \cite{Chen13a,Chen13b}, the method based on single layer
potentials \cite{Stolk}, and the method of polarized traces \cite{ZD},
are all methods in this same class of domain decomposition methods
with more effective transmission conditions, which became known under
the name optimized Schwarz methods, see
\cite{gander2006optimized,gander2008schwarz} for an introduction, and
\cite{ganderzhang2018SIREV} and references therein for a thorough
treatment when applied to time harmonic wave propagation problems.

To the best of our knowledge, the use of Schwarz methods for
time-harmonic elastic waves modeled by the Navier equations has not
been studied so far, and our goal is to investigate classical Schwarz
methods, and also a new variant that uses more natural transmission
conditions between the subdomains when applied to the Navier
equations. To do so, we study the Schwarz methods at the continuous
level, for a simplified decomposition as it has become standard with
two subdomains, to gain insight into the effect of transmission
conditions on the performance of the method. To test the method, we
then discretize the problems and implement the Schwarz methods using
Restricted Additive Schwarz (RAS) introduced by Cai and
Sarkis in \cite{Cai:99:RAS}, which represents a faithful
implementation of the continuous parallel Schwarz method of Lions, see
\cite{gander2008schwarz}. This is especially important when more
natural transmission conditions are used, see \cite{Stcyr:07:OMA} for
Optimized RAS (ORAS).

Our paper is structured as follows: in Section \ref{sec:classical}, we
present and analyze the classical Schwarz algorithm applied to the
Navier equations. We prove for a simplified two subdomain setting at
the continuous level that the Schwarz algorithm is not a convergent
iterative method in this case. We then introduce new transmission
conditions in Section \ref{sec:optimal} and show first that there
exist transmission conditions which make the Schwarz method converge
in a finite number of steps. These transmission conditions involve
however non local operators, and we thus introduce a local, low
frequency approximation for the Navier equations, for which we prove
convergence of the new Schwarz method provided the overlap is not too
small. In Section \MG{\ref{sec:numerical}} we study these new Schwarz methods
numerically, first for a two subdomain decomposition covered by our
analysis, but then also for the case of many subdomains with cross
points and material heterogeneities. Our numerical results show that
the new Schwarz method performs much better than the classical one
when used as a preconditioner for a Krylov method.
	
\section{Classical Schwarz algorithm for the Navier Equations}
\label{sec:classical}

We are interested in solving the Navier equations in the frequency
domain,
\begin{equation}\label{NavierEq}
  -\left(\Delta^e+\omega^2\rho\right)\vec{u}=\vec{f}\quad \mbox{in $\Omega$},
\end{equation}
where the operator $\Delta^e$ is defined by
$\Delta^e\vec{u}=\mu\Delta\vec{u}+(\lambda+\mu)\nabla(\nabla\cdot\vec{u})$.
To study the basic (non)-convergence properties of the Schwarz
algorithm applied to the Navier equations \eqref{NavierEq}, we
consider the domain $\Omega:={\mathbb R}^2$ and decompose it into two
unbounded overlapping subdomains
$\Omega_1:=(-\infty,\delta)\times{\mathbb R}$ and
$\Omega_2:=(0,\infty)\times{\mathbb R}$, with overlap parameter
$\delta > 0$.  The classical parallel Schwarz algorithm then starts
with an initial guess $\vec{u}_j^0$ on subdomain $\Omega_j$, $j=1,2$,
and solves for iteration index $n=1,2,\ldots$
\begin{equation}\label{ClassicalSchwarz}
  \arraycolsep0.3em
  \begin{array}{rcllrcll}
    -\left(\Delta^e+\omega^2\rho\right)\vec{u}_1^n&=&\vec{f} &\mbox{in $\Omega_1$,}&
        -\left(\Delta^e+\omega^2\rho\right)\vec{u}_2^n&=&\vec{f} &\mbox{in $\Omega_2$,}\\
    \vec{u}_1^n&=&
    \vec{u}_2^{n-1} & \mbox{at $x=\delta$},& 
    \vec{u}_2^n&=&
    \vec{u}_1^{n-1} & \mbox{at $x=0$}.
  \end{array}
\end{equation}
To study the convergence properties of this algorithm, we use a
Fourier transform in the $y$ direction. We denote by $k \in
\mathbb{R}$ the Fourier parameter and $\hat{u}(x,k)$ the Fourier
transformed solution,
$$
\hat{\vec{u}}(x,k) = \int_{-\infty}^{\infty} \expo^{-\mathrm{i} k y} \vec{u}(x,y) \, \mathrm dy, \quad
\vec{u}(x,y) = \frac{1}{2\pi} \int_{-\infty}^{\infty} \expo^{\mathrm{i} k y} \hat{\vec{u}}(x,k) \, \mathrm dk.
$$
The convergence factor for each Fourier mode of
(\ref{ClassicalSchwarz}) is given in 
\begin{lemma}[Convergence factor of classical Schwarz]
\label{th:convclas}
For a given initial guess $\vec{u}_j^0 \in(L^2(\Omega_j))^2$, $j=1,2$,
the classical Schwarz algorithm \eqref{ClassicalSchwarz} with overlap
$\delta>0$ multiplies at each iteration the error in each Fourier mode $k$
with the convergence factor
\begin{equation}\label{RhoClassical}
 \rho_{cla}\left(k,\omega,C_p,C_s,\delta\right) = \max \{|r_+|,|r_-|\},
\end{equation}
where \MG{the eigenvalues of the iteration matrix are}
\begin{equation}\label{r+r-}
r_\pm = \frac{X^2}{2}+e^{-\delta(\lambda_1+\lambda_2)} \pm\frac{1}{2} \sqrt{X^2\left(X^2+4e^{-\delta(\lambda_1+\lambda_2)}\right)},\,
  X = \frac{k^2+\lambda_1\lambda_2}{k^2-\lambda_1\lambda_2}\left(e^{-\lambda_1 \delta}-e^{-\lambda_2 \delta}\right)\MG{,}
\end{equation}
and $\lambda_{1,2} \in \mathbb{C}$ are given by
\begin{equation}  \label{lambda12}
  \lambda_1 = \sqrt{k^2-\frac{\omega^2}{C_s^2}},\quad 
  \lambda_2 = \sqrt{k^2-\frac{\omega^2}{C_p^2}},\quad
C_p=\sqrt{\frac{\lambda+2\mu}{\rho}}, \quad
C_s=\sqrt{\frac{\mu}{\rho}}.
\end{equation}
\end{lemma}
\begin{proof}
  The convergence factor can be obtained by a direct computation
  working on the error equations, as it is shown in 
  the short publication \cite{Brunet:2019:CCS}.
\end{proof}

We show in Figure \ref{Fig3a} a plot of the modulus of the
convergence factor \eqref{RhoClassical} as function of the Fourier
mode $k$ for an example of the parameters in the Navier equations.
\begin{figure}
  \centering  
  \includegraphics[width=0.5\textwidth]{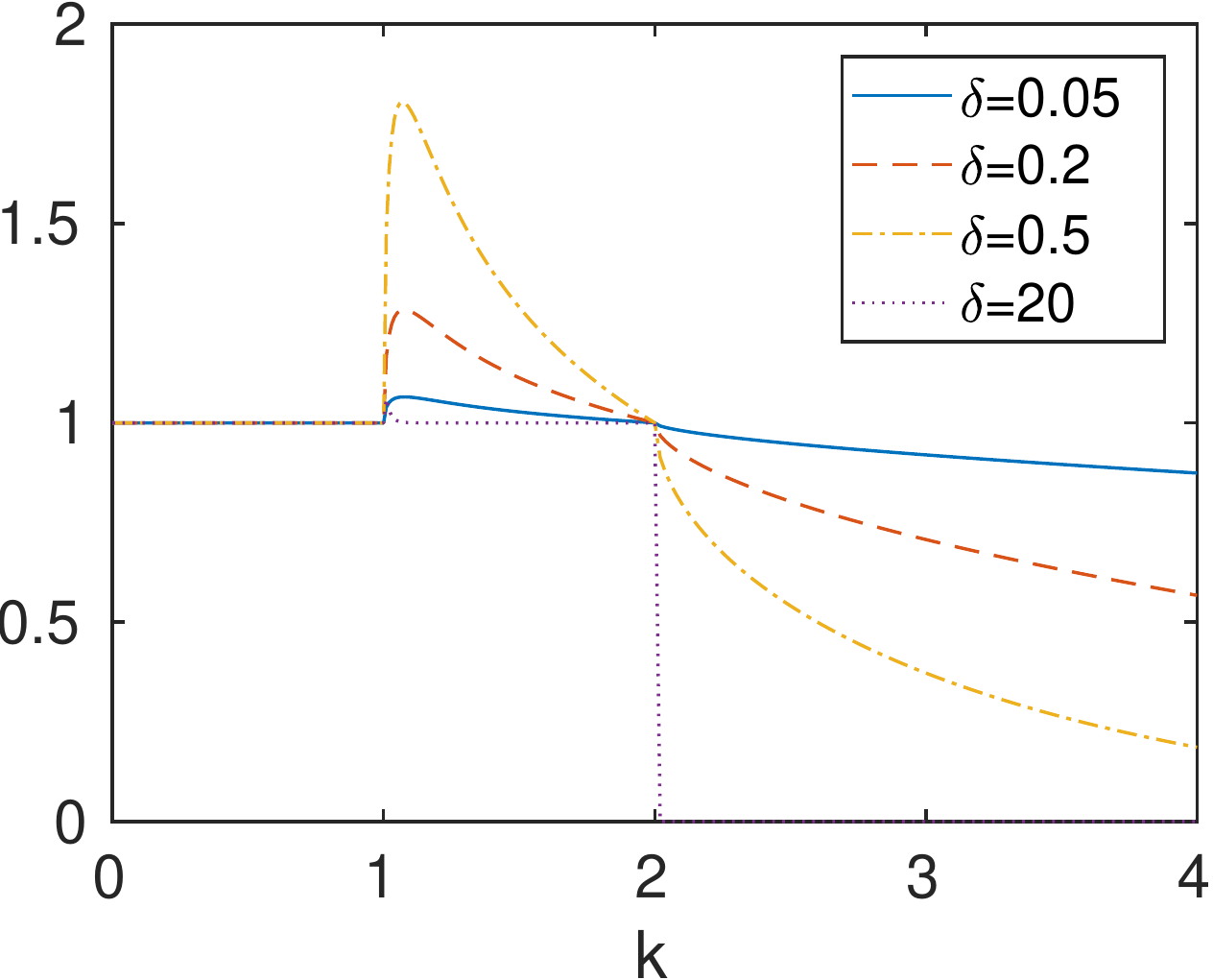}
  \caption{Modulus of the convergence factor of the classical Schwarz
    method for $C_p=1$, $C_s=\frac{1}{2}$, $\omega=1$ for
    different values of the overlap $\delta$.}
  \label{Fig3a}
\end{figure}
We see that the classical Schwarz method converges for high frequencies,
$|\rho_{cla}|<1$, diverges for intermediate frequencies,
$|\rho_{cla}|>1$, and stagnates for low frequencies
$|\rho_{cla}|=1$. We prove in the next theorem that this behavior
holds for all choices of parameters in the Navier equations, and
thus the classical Schwarz method is not an effective iterative solver for
these equations.
\begin{theorem}[(Non-) Convergence of the overlapping classical Schwarz method]
\label{ClassicalSchwarzTheorem}
The convergence factor (\ref{ClassicalSchwarz}) of the overlapping
classical Schwarz method \eqref{ClassicalSchwarz} applied to the
Navier equations (\ref{NavierEq}) satisfies
\begin{equation}
\begin{array}{rcll}
|\rho_{cla}(k,\omega,C_p,C_s,\delta)| 
& = & 1,& k\in[0,\frac{\omega}{C_p}]\cup\{\frac{\omega}{C_s}\}, \\ 
|\rho_{cla}(k,\omega,C_p,C_s,\delta)| 
&>& 1,& k\in (\frac{\omega}{C_p},\frac{\omega}{C_s}), \\
|\rho_{cla}(k,\omega,C_p,C_s,\delta)| 
&<& 1,& k\in (\frac{\omega}{C_s},\infty),
\end{array}
\end{equation}
where the last two results are shown to hold for overlap $\delta$ small.
\end{theorem}
\begin{proof}
The proof is quite technical. To simplify the notation, we define for
the case when the roots $\lambda_{1,2}$ in (\ref{lambda12}) are complex
the quantities
\begin{equation}
  \textstyle \mathrm{i} \bar{\lambda}_{1}:=\lambda_1 =
    \mathrm{i}  \sqrt{\frac{\omega^2}{C_s^2} - k^2}, \quad
  \mathrm{i} \bar{\lambda}_{2}:=\lambda_2 =
    \mathrm{i} \sqrt{\frac{\omega^2}{C_p^2} - k^2}.
\label{ImagLambda}
\end{equation}
We have to treat five cases: three intervals for $k$, and two values
$k\in \{\frac{\omega}{C_p},\frac{\omega}{C_s}\}$ separating the
intervals: in the first interval $k
\in(0,\frac{\omega}{C_p})$, $\lambda_{1,2} \in \mathrm{i}
\mathbb{R_+}$, and the eigenvalues (\ref{r+r-}) become
$$
  \textstyle r_\pm = \frac{X^2}{2} +
  \expo^{- \mathrm{i} \delta(\bar{\lambda}_1 + \bar{\lambda}_2)}
  \pm \frac{1}{2} \sqrt{X^2\left(X^2+4\expo^{- \mathrm{i} \delta \left( \bar{\lambda}_1+\bar{\lambda}_2\right)}\right)},
X =\frac{k^2 - \bar{\lambda}_1 \bar{\lambda}_2}{k^2 + \bar{\lambda}_1 \bar{\lambda}_2} \left(\expo^{- \mathrm{i} \bar{\lambda}_1 \delta}-\expo^{-\mathrm{i} \bar{\lambda}_2 \delta}\right).
$$
The square of their modulus is given by
\begin{equation}
\begin{aligned}
&|r_\pm|^2 = 1 + \underbrace{\textstyle\frac{\sqrt{A^2+B^2}+(x^2+y^2)^2}{4} + e_r \left(x^2-y^2\right) + 2xye_i}_{\Part_1} \:\pm\: \frac{\sqrt{2}}{2} \:\times\: \\
&\left( \underbrace{\textstyle \frac{x^2 - y^2 + 2 e_r}{2} \left(\sqrt{A^2+B^2}+A\right)^{\frac{1}{2}} + \csgn\left(B-\mathrm{i}A\right) (xy+e_i) \left(\sqrt{A^2+B^2}-A\right)^{\frac{1}{2}}}_{\Part_2} \right),
\end{aligned}
\label{rcase1}
\end{equation}
where the complex sign is defined as
$$
\csgn(x)=\left\lbrace
\begin{aligned}
 1 \qquad &0<\mathfrak{R}(x) \quad\mbox{or}\quad \mathfrak{R}(x)=0 \;\&\; 0<\mathfrak{I}(x),\\
 -1 \qquad &\mathfrak{R}(x)>0 \quad\mbox{or}\quad \mathfrak{R}(x)=0 \;\&\; \mathfrak{I}(x)>0,
\end{aligned}
\right.
$$
and we introduced the quantities $e_r$, $e_i$, $x$ and $y$, 
$$
\begin{aligned}
  e_r &:=  -\sin\left(\delta\left(\bar{\lambda}_1+\bar{\lambda}_2\right)\right)
  ,\quad
  e_i := \cos\left(\delta\left(\bar{\lambda}_1+\bar{\lambda}_2\right)\right)
  ,\\
  x &:= \Re{(X)} = \textstyle \frac{k^2-\bar\lambda_1 \bar\lambda_2}{k^2+\bar{\lambda}_1 \bar{\lambda}_2} \left(\cos\left(\bar\lambda_1 \delta\right)-\cos\left(\bar\lambda_2 \delta\right)\right),\\
  y &:= \Im{(X)} = \textstyle-\frac{k^2-\bar\lambda_1 \bar\lambda_2 }{k^2+\bar\lambda_1 \bar\lambda_2}\left(\sin\left(\bar\lambda_1 \delta\right) - \sin\left(\bar\lambda_2 \delta\right)\right)\MG{.}
\end{aligned}
$$
\MG{The} terms $A$ and $B$ appearing in the square root are real and
defined by $A+\mathrm{i}B := X^2\left(X^2+4\expo^{- \mathrm{i} \delta
  \left(\bar{\lambda}_1+\bar{\lambda}_2\right)}\right)$, which gives
after some computations
$$
\begin{aligned}
A&= \left(x^2-y^2\right)^2-4x^2y^2 - 8e_ixy + 4e_r\left(x^2-y^2\right),\\
B&= 4(xy+e_i)\left(x^2-y^2\right) + 8e_rxy.
\end{aligned}
$$
Then we obtain by a direct computation that 
$$
\begin{aligned}
&\sqrt{A^2+B^2}=\left(x^2+y^2\right)\sqrt{\left(x^2+y^2\right)^2+8 e_r\left(x^2-y^2\right)+16 e_i x y+16} \\
&= \textstyle \frac{16 \left(k^2-\bar{\lambda}_1 \bar{\lambda}_2\right)^2\sin^2\left(\frac{\delta}{2} \left(\bar{\lambda}_1-\bar{\lambda}_2\right)\right)}{\left(k^2+\bar{\lambda}_1 \bar{\lambda}_2\right)^2} \\
&\times \textstyle \left(1-\frac{2 \left(k^2-\bar{\lambda}_1 \bar{\lambda}_2\right)^2 \sin^2\left(\frac{\delta}{2} \left(\bar{\lambda}_1-\bar{\lambda}_2\right)\right)}{\left(k^2+\bar{\lambda}_1 \bar{\lambda}_2\right)^2} + \frac{\left(k^2-\bar{\lambda}_1 \bar{\lambda}_2\right)^4 \sin^4\left(\frac{\delta}{2} \left(\bar{\lambda}_1-\bar{\lambda}_2\right)\right)}{\left(k^2+\bar{\lambda}_1 \bar{\lambda}_2\right)^4}\right) \\
&= \textstyle \frac{16 \sin^2\left(\frac{\delta}{2} \left(\bar{\lambda}_1-\bar{\lambda}_2\right)\right) \left(\left(k^2-\bar{\lambda}_1 \bar{\lambda}_2\right)^2 \cos^2\left(\frac{\delta}{2} \left(\bar{\lambda}_1-\bar{\lambda}_2\right)\right)+4\bar{\lambda}_1 k^2 \bar{\lambda}_2\right)}{\left(k^2-\bar{\lambda}_1 \bar{\lambda}_2\right)^{-2}\left(k^2+\bar{\lambda}_1 \bar{\lambda}_2\right)^4},\\
\end{aligned}
$$
and
$$
\begin{aligned}
x^2-y^2 &= \textstyle- \frac{4 \left(k^2-\bar{\lambda}_1 \bar{\lambda}_2\right)^2 \sin^2\left(\frac{\delta}{2} \left(\bar{\lambda}_1-\bar{\lambda}_2\right)\right) \cos\left(\delta \left(\bar{\lambda}_1+\bar{\lambda}_2\right)\right)}{\left(k^2+\bar{\lambda}_1 \bar{\lambda}_2\right)^2}, \\ 
x^2+y^2 &= \textstyle\frac{4 \left(k^2-\bar{\lambda}_1 \bar{\lambda}_2\right)^2 \sin^2\left(\frac{\delta}{2}\left(\bar{\lambda}_1-\bar{\lambda}_2\right)\right)}{\left(k^2+\bar{\lambda}_1 \bar{\lambda}_2\right)^2},\\ 
xy &= \textstyle \frac{2 \left(k^2-\bar{\lambda}_1 \bar{\lambda}_2\right)^2 \sin\left(\delta \left(\bar{\lambda}_1+\bar{\lambda}_2\right)\right) \sin^2\left(\frac{\delta}{2}\left(\bar{\lambda}_1-\bar{\lambda}_2\right)\right)}{\left(k^2+\bar{\lambda}_1\bar{\lambda}_2\right)^2}.
\end{aligned}
$$
We now show that $\Part_1$ in \eqref{rcase1} vanishes identically:
we get on the one hand
\begin{equation}
\frac{\left(x^2+y^2\right)^2}{4} = 4\frac{\left(k^2-\bar{\lambda}_1 \bar{\lambda}_2\right)^4 \sin^4\left(\frac{\delta}{2}\left(\bar{\lambda}_1- \bar{\lambda}_2\right)\right)}{\left(k^2+\bar{\lambda}_1 \bar{\lambda}_2\right)^4},
\label{part1a}
\end{equation}
and on the other hand, we have
\begin{equation}
\begin{aligned}
\textstyle \frac{\sqrt{A^2+B^2}}{4} &= \textstyle \frac{\left(k^2-\bar{\lambda}_1 \bar{\lambda}_2\right)^4 \sin^2\left(\delta \left(\bar{\lambda}_1-\bar{\lambda}_2\right)\right)}{\left(k^2+\bar{\lambda}_1 \bar{\lambda}_2\right)^4} + \frac{ 16\sin^2\left(\frac{\delta}{2} \left(\bar{\lambda}_1 -\bar{\lambda}_2\right)\right) \bar{\lambda}_1 k^2 \bar{\lambda}_2}{\left(k^2-\bar{\lambda}_1 \bar{\lambda}_2\right)^{-2}\left(k^2+ \bar{\lambda}_1 \bar{\lambda}_2\right)^4}, \\
e_r\left(x^2-y^2\right) &= \textstyle- \frac{4 \left(k^2-\bar{\lambda}_1\bar{\lambda}_2\right)^2 \sin^2\left(\frac{\delta}{2} \left(\bar{\lambda}_1-\bar{\lambda}_2\right)\right) \cos^2\left(\delta \left(\bar{\lambda}_1+\bar{\lambda}_2\right)\right)}{\left(k^2+\bar{\lambda}_1\bar{\lambda}_2\right)^2},\\
2e_ixy &= \textstyle-\frac{4 \left(k^2-\bar{\lambda}_1\bar{\lambda}_2\right)^2 \sin^2\left(\delta \left(\bar{\lambda}_1+\bar{\lambda}_2\right)\right) \sin^2\left(\frac{\delta}{2} \left(\bar{\lambda}_1-\bar{\lambda}_2\right)\right)}{\left(k^2+\bar{\lambda}_1\bar{\lambda}_2\right)^2},
\end{aligned}
\label{part1b}
\end{equation}
and we obtain by adding the three terms from (\ref{part1b}) to each other
\begin{equation}
-\frac{4 \left(k^2-\bar{\lambda}_1\bar{\lambda}_2)^4 \sin^4\left(\frac{\delta}{2}(\bar{\lambda}_1-\bar{\lambda}_2\right)\right)}{\left(k^2+\bar{\lambda}_1\bar{\lambda}_2\right)^4}.
\label{part1c}
\end{equation}
This leads, by adding (\ref{part1a}) and (\ref{part1c}) indeed to 
$\Part_1\equiv 0$. We next show that also $\Part_2$ in \eqref{rcase1}
vanishes identically: we get
$$
\begin{aligned}
&\textstyle\frac{x^2 - y^2}{2}+ e_r = \cos\left(\delta\left(\bar{\lambda}_1+\bar{\lambda}_2\right)\right) \left( 1 - 2 \frac{\left(k^2-\bar{\lambda}_1\bar{\lambda}_2\right)^2 \sin^2\left(\frac{\delta}{2} \left( \bar{\lambda}_1-\bar{\lambda}_2\right)\right)}{\left(k^2+ \bar{\lambda}_1\bar{\lambda}_2\right)^2} \right), \\
&\textstyle xy+e_i = - \sin\left(\delta\left(\bar{\lambda}_1+\bar{\lambda}_2\right)\right)\left(1-2\frac{\left(k^2-\bar{\lambda}_1\bar{\lambda}_2\right)^2 \sin^2\left(\frac{\delta}{2}\left(\bar{\lambda}_1-\bar{\lambda}_2\right)\right)}{\left(k^2+\bar{\lambda}_1\bar{\lambda}_2\right)^2} \right),
\end{aligned}
$$
and for the term involving $A$ and $B$
$$
\begin{aligned}
\sqrt{\sqrt{A^2+B^2} \pm A} &= \textstyle 4\frac{k^2-\bar{\lambda}_1\bar{\lambda}_2}{\left(k^2+\bar{\lambda}_1\bar{\lambda}_2\right)^2} \sin\left(\frac{\delta}{2}\left(\bar{\lambda}_1-\bar{\lambda}_2\right)\right) \sqrt{1 \mp \cos\left(2\delta\left(\bar{\lambda}_1+\bar{\lambda}_2\right)\right)} \\
&\textstyle\times \sqrt{\left(k^2-\bar{\lambda}_1\bar{\lambda}_2\right)^2 \cos^2\left(\frac{\delta}{2}\left(\bar{\lambda}_1-\bar{\lambda}_2\right)\right)+ 4 k^2 \bar{\lambda}_1 \bar{\lambda}_2}.
\end{aligned}
$$
By analyzing the signs of the different terms, we obtain for the complex sign
$$\csgn(B-\mathrm{i}A)=\sg\left(\cos\left(\delta\left(\bar{\lambda}_1+\bar{\lambda}_2\right)\right)\sin\left(\delta\left(\bar{\lambda}_1+\bar{\lambda}_2\right)\right)\right),$$
and after a lengthy computation we obtain
$$
\begin{aligned}
\Part_2 = C_k\times &\left( \sqrt{1+\cos\left(2\delta\left(\bar{\lambda}_1+\bar{\lambda}_2\right)\right)} \sin\left(\delta\left(\bar{\lambda}_1+\bar{\lambda}_2\right)\right) \right.\\
&\left.-\csgn\left(B-\mathrm{i}A\right)\cos\left(\delta\left(\bar{\lambda}_1+\bar{\lambda}_2\right)\right)\sqrt{1-\cos\left(2\delta\left(\bar{\lambda}_1+\bar{\lambda}_2\right)\right)}\right),
\end{aligned}
$$
where $C_k\in\MG{\mathbb{R^*}:=\mathbb{R}\backslash\left\{0\right\}}$ is a complicated factor depending on $k$.
A direct computation for the second factor of $\Part_2$ shows that
independently of the value of $\csgn(B-\mathrm{i}A)$, we get $\Part_2
\equiv 0$.  We can thus conclude from (\ref{rcase1}) that $
\rho_{cla}\left(k,\omega,C_p,C_s,\delta\right) = \max\{|r_+|,|r_-|\}=
|r_+|=|r_-|= 1$ and therefore the algorithm stagnates in the first
interval $k\in [0,\frac{\omega}{C_p})$, see the first interval
in Figure \ref{Fig3a}.

At the boundary between the first and second interval, where $k =
\frac{\omega}{C_p}$, we have that $\lambda_{2}=0$ and
$\lambda_{1}\in\mathrm{i}\mathbb{R^*_+}$, and therefore \MG{the eigenvalues in} (\ref{r+r-})
become
$$
r_\pm = \frac{1}{2}(1+\expo^{-2\mathrm{i}\bar{\lambda}_1\delta}) \pm \frac{1}{2}\sqrt{\left(1-\expo^{-2\mathrm{i}\bar{\lambda}_1\delta}\right)^2},\quad X = \expo^{-\mathrm{i}\bar{\lambda}_1\delta}-1,
$$
and $\mathfrak{Re}(1-\expo^{-2\mathrm{i}\bar{\lambda}_1\delta})=1-\cos(2\bar{\lambda}_1\delta)$ being positive we have equivalently
$$
r_+ = 1,\ r_- = \expo^{-2\mathrm{i}\bar{\lambda}_1\delta} \quad\Longrightarrow\quad
\rho_{cla}(\textstyle\frac{\omega}{C_p},\omega,C_p,C_s,\delta) = \max\left\{|r_+|,|r_-|\right\} = 1,
$$
and hence the algorithm stagnates also when the first
interval is closed on the right, i.e. for $k\in [0,\frac{\omega}{C_p}]$.

In the second interval, $k \in
(\frac{\omega}{C_p},\frac{\omega}{Cs})$, we have that
$\lambda_{1}\in\mathrm{i}\mathbb{R^*_+}$ and $
\lambda_{2}\in\mathbb{R^*_+}$, and hence \MG{the eigenvalues in}
(\ref{r+r-}) become
$$
\textstyle r_\pm = \frac{X^2}{2} + \expo^{- \delta\left(\mathrm{i} \bar{\lambda}_1 +\lambda_2\right)} \pm \frac{1}{2}\sqrt{X^2(X^2+ 4 \expo^{-\delta\left(\mathrm{i} \bar{\lambda}_1 + \lambda_2\right)})},\,
X = \frac{k^2+ \mathrm{i} \bar{\lambda}_1 \lambda_2}{k^2- \mathrm{i} \bar{\lambda}_1 \lambda_2}(\expo^{-\mathrm{i} \bar{\lambda}_1 \delta}-\expo^{-\lambda_2\delta}).
$$
We compute the modulus of the eigenvalues and expand them
for overlap parameter $\delta$ small to find
$$
\textstyle
  {|r_+|}= 1 +  \frac{2 \omega^2 \lambda_2 \bar{\lambda}_1^2}{C_p^2\left(k^4+\bar{\lambda}_1^2 \lambda_2^2\right)} \delta + \mathcal{O}(\delta^2), \quad
{|r_-|}= 1 - \frac{2 \omega^2\lambda_2 k^2}{C_s^2\left(k^4 + \bar{\lambda}_1^2 \lambda_2^2\right)} \delta + \mathcal{O}(\delta^2).
$$
We thus obtain that $\rho_{cla}(k,\omega,C_p,C_s,\delta) =
\max\{|r_+|,|r_-|\}$ is bigger than one for $\delta$ small and the
method diverges, see the middle interval in Figure
\ref{Fig3a}\footnote{Numerically we observe that also for a large
  overlap, the algorithm diverges, see Figure \ref{Fig3a}, but
  this seems to be difficult to prove.}.

Between the second and third interval, where $k = \frac{\omega}{C_s}$,
we have that $\lambda_{1}=0$ and $\lambda_{2} =
\frac{\omega\sqrt{C_p^2-C_s^2}}{C_s C_p}>0$, and hence \MG{the
  eigenvalues in} (\ref{r+r-}) become
$$
r_\pm = \frac{1}{2}\left(1+\expo^{- 2\lambda_2\delta}\right) \pm \frac{1}{2}\sqrt{\left(1-\expo^{-2\lambda_2\delta}\right)^2}.
$$
We thus obtain
$$
  r_+ = 1,\ r_- = \expo^{-2\lambda_2\delta} \quad \Longrightarrow\quad
  \rho_{cla}(\textstyle\frac{\omega}{C_s},\omega,C_p,C_s,\delta)
  = \max\{|r_+|,|r_-|\} = 1,
$$
and the algorithm stagnates for $k = \frac{\omega}{C_s}$.
  
In the last interval, $k \in \left(\frac{\omega}{Cs},\infty\right)$,
$\lambda_{1,2}\in\mathbb{R^*_+}$ and by expanding $r_\pm>0$ from
\eqref{r+r-} for $\delta$ small, we get
$$
 \textstyle {r_+}= 1 - \frac{2 \lambda_2 \omega^2}{C_s^2(k^2-\lambda_1 \lambda_2)}\delta + \mathcal{O}(\delta^2) \overset{}{<} 1,\quad
  {r_-}= 1 - \frac{2 \lambda_1 \omega^2}{C_p^2(k^2-\lambda_1 \lambda_2)}\delta + \mathcal{O}(\delta^2) < 1,
$$
since $k^2-\lambda_1 \lambda_2 >0$. We can thus conclude that
$$
\rho_{cla}(k,\omega,C_p,C_s,\delta) = \max\{|r_+|,|r_-|\} < 1,
$$
see the last interval in Figure \ref{Fig3a}, where we also
see that $\lim_{k\rightarrow \infty} r_\pm = 0$, since all the real
exponentials involved in the expressions of $r_{\pm}$ are decreasing
to $0$ as $k$ increases.
\end{proof}

We see from Theorem \ref{ClassicalSchwarzTheorem} that the classical
Schwarz method with overlap can not be used as an iterative solver to
solve the Navier equations, since the method stagnates for low
frequencies and even diverges for intermediate frequencies; only high
frequencies are converging. A precise estimate for how fast the
classical Schwarz method diverges depending on the overlap is
given in the short publication \cite{Brunet:2019:CCS}.

\section{New Transmission Conditions for the Schwarz algorithm}
\label{sec:optimal}

A remedy for the divergence problems of the classical Schwarz method is
to introduce different transmission conditions, and to consider the
new Schwarz method
\begin{equation}\label{GeneralOptimizedSchwarz}
  \arraycolsep0.15em
  \begin{array}{rcllrcll}
  -\left(\Delta^e+\omega^2\rho\right)\vec{u}_1^n&=&\vec{f} &
    \mbox{in $\Omega_1$,}&
    -\left(\Delta^e+\omega^2\rho\right)\vec{u}_2^n&=&\vec{f}
    &\mbox{in $\Omega_2$,}\\
    \left({\cal T}_1+{\cal S}_1\right)\vec{u}_1^n&=&
    \left({\cal T}_1+{\cal S}_1\right)\vec{u}_2^{n-1} &x=\delta,& 
    \left({\cal T}_2+{\cal S}_2\right)\vec{u}_2^n&=&
    \left({\cal T}_2+{\cal S}_2\right)\vec{u}_1^{n-1},& x=0,
  \end{array}
\end{equation}
where the traction operators ${\cal T}_j$, $j=1,2$, are defined by
$T_j(\vec{u})=2\mu\frac{\partial \vec{u}}{\partial
  n_j}+\lambda\vec{n}_j\nabla\cdot\vec{u}
+\mu\vec{n}_j\times\nabla\times\vec{u}$, and the operators ${\cal
  S}_j$ are two by two matrix valued operators one can choose to
obtain better convergence. The traction operators ${\cal T}_j$ play
for the Navier equations the role the Neumann condition plays for the
Poisson equation. Like we obtained the convergence factor of the
classical Schwarz algorithm using a Fourier transform in Lemma
\ref{th:convclas}, we can obtain the convergence factor in the case
where more general transmission operators ${\cal S}_{1,2}$
with Fourier symbols $\widehat{\cal S}_{1,2}$ are used.
\begin{lemma}\label{general_opti}
For a given initial guess $\vec{u}_j^0 \in(L^2(\Omega_j))^2$, $j=1,2$,
the general Schwarz algorithm with overlap
(\ref{GeneralOptimizedSchwarz}) has for each Fourier mode the
convergence factor
\begin{equation}
 \rho_{opt}(k,\omega,C_p,C_s,\delta) = \left(\max \{|r_+|,|r_-|\}\right)^{\frac{1}{2}}, \quad
 r_\pm = \frac{X^2}{2}+ Y \pm \frac{1}{2} \sqrt{X^2(X^2+4Y)},
 \label{OptiCv}
\end{equation}
with
\begin{equation}
X = \expo^{-\lambda_1 \delta}b_{11} - \expo^{-\lambda_2 \delta}b_{22}, \quad
Y = \frac{b_{11}b_{22}-b_{12}b_{21}}{\expo^{\lambda_1 \delta}\expo^{\lambda_2 \delta}}, \quad
\begin{bmatrix}
b_{11}  & b_{12} \\
b_{21}  & b_{22}
\end{bmatrix} := B_2^{-1}B_1,
\label{B1B2}
\end{equation}
where
\begin{equation}
B_1 =
\begin{bmatrix}
\widehat{\cal S}_2(1,1) - 2 \lambda_1 \rho C_s^2 - \mathrm{i}  \frac{  \lambda_1 \widehat{\cal S}_2(1,2)}{k}
& \widehat{\cal S}_2(1,2) + \mathrm{i} \frac{2k^2\rho C_s^2 - \lambda_2 \widehat{\cal S}_2(1,1) - \rho\omega^2}{k}\\
 \widehat{\cal S}_2(2,1) - \mathrm{i} \frac{2 k^2 C_s^2\rho - \lambda_1 \widehat{\cal S}_2(2,2) - \rho\omega^2}{k}
& \widehat{\cal S}_2(2,2) + 2 C_s^2 \rho \lambda_2 + \mathrm{i} \frac{ \lambda_2 \widehat{\cal S}_2(2,1)}{k}
\end{bmatrix},
\label{OptiB1}
\end{equation}
\begin{equation}
B_2 = 
\begin{bmatrix}
\widehat{\cal S}_2(1,1) + 2 \lambda_1 \rho C_s^2 + \mathrm{i}  \frac{  \lambda_1 \widehat{\cal S}_2(1,2)}{k}
& \widehat{\cal S}_2(1,2) + \mathrm{i} \frac{2k^2\rho C_s^2 + \lambda_2 \widehat{\cal S}_2(1,1) - \rho\omega^2}{k}\\
 \widehat{\cal S}_2(2,1) - \mathrm{i} \frac{2 k^2 C_s^2\rho + \lambda_1 \widehat{\cal S}_2(2,2) - \rho\omega^2}{k}
& \widehat{\cal S}_2(2,2) - 2 C_s^2 \rho \lambda_2 - \mathrm{i} \frac{ \lambda_2 \widehat{\cal S}_2(2,1)}{k}
\end{bmatrix},
\label{OptiB2}
\end{equation}
and $\lambda_{1,2} \in \mathbb{C}$ are given in \eqref{lambda12}.
\end{lemma}
\begin{proof}
  This result is obtained by a direct calculation, replacing the
  solutions in Fourier space into the transmission conditions of the
  general Schwarz algorithm \eqref{GeneralOptimizedSchwarz}, for
  details, see the PhD thesis \cite[Lemma 2.3]{BrunetPhD2018}.
\end{proof}

\subsection{An Optimal Schwarz Method}

The new transmission conditions in (\ref{GeneralOptimizedSchwarz}) are
a very powerful tool to fix convergence problems of the classical
Schwarz method, and are used in many modern domain decomposition
methods for time harmonic wave propagation, like the sweeping
preconditioner, source transfer and the method of polarized traces,
which are all variants of the so called optimized Schwarz methods
\cite{gander2006optimized,gander2008schwarz}; for a review, see
\cite{ganderzhang2018SIREV}. To see how powerful this idea is, we
start by introducing the best possible choice, namely transparent
boundary conditions (TBC) as transmission conditions in
(\ref{GeneralOptimizedSchwarz}), which leads to what is called an {\em
  optimal Schwarz method}\footnote{Optimal here is not used in the sense
of scalability, but really means faster convergence is not possible!}:
\begin{theorem}[Convergence of the optimal Schwarz algorithm.]
  \label{OptimalSchwarzThm}
  If one chooses in the new Schwarz algorithm
  (\ref{GeneralOptimizedSchwarz}) the operators ${\cal S}_j$ with the
  Fourier symbols
\begin{equation}\label{OptimalChoice}
\begin{array}{rcl}
\widehat{\cal S}_1(1,1)&=& \rho \frac{\lambda_1 \omega^2}{k^2-\lambda_1\lambda_2}, \\ 
\widehat{\cal S}_1(1,2)&=& +\mathrm{i} k \rho (2C_s^2 -  \frac{\omega^2}{k^2 - \lambda_1 \lambda_2}),\\ 
\widehat{\cal S}_1(2,1) &=& -\mathrm{i} k \rho (2C_s^2 -  \frac{\omega^2}{k^2 - \lambda_1\lambda_2}), \\ 
\widehat{\cal S}_1(2,2) &=& \rho \frac{\lambda_2 \omega^2}{k^2-\lambda_1 \lambda_2}, 
\end{array}
\quad
\begin{array}{rcl}
\widehat{\cal S}_2(1,1)&=& \widehat{\cal S}_1(1,1), \\ 
\widehat{\cal S}_2(1,2)&=& -\widehat{\cal S}_1(1,2), \\
\widehat{\cal S}_2(2,1)&=& -\widehat{\cal S}_1(2,1), \\
\widehat{\cal S}_2(2,2)&=& \widehat{\cal S}_1(2,2),
\end{array}
\end{equation}
where $\lambda_1$ and $\lambda_2$ are given in \eqref{lambda12}, the
resulting algorithm converges in two iterations, and this for all
values of the overlap $\delta\ge 0$, even without overlap, $\delta=0$.
\end{theorem}
\begin{proof}
  If we replace $(\widehat{\cal S}_1,\widehat{\cal S}_2$) defined in
  \eqref{OptimalChoice} into \eqref{OptiB1}, the convergence factor
  obtained vanishes identically and the algorithm thus converges in
  two iterations, independently of any initial guess and the overlap
  $\delta\ge 0$.
\end{proof}

To use the optimal choice \eqref{OptimalChoice} as transmission
operators in practice, one needs to back transform the associated TBC
into the physical domain, and the corresponding ${\cal S}_j$ are non
local operators, because of the inverse transform with square root
terms at the interfaces, like it is the case for many TBCs.  It is
therefore of interest to design local approximations for the optimal
transmissions conditions, like in the development of absorbing
boundary conditions (ABCs), which will lead to a new class of
practical, so called optimized Schwarz algorithms. We approximate the
symbols of the optimal transmission conditions in
\eqref{OptimalChoice} by polynomial symbols in $\mathrm{i}k$ which
correspond to derivatives after the Fourier backtransform, and are
thus local operators.

\subsection{Optimized Schwarz Methods}
\label{sec:absorbing}

We have seen in Section \ref{sec:classical} that the classical Schwarz
method converges well for high frequency error components, $k$ large,
but stagnates for low frequency components and even diverges for
intermediate range frequencies, see Figure \ref{Fig3a}.  It is
therefore natural to approximate the operators ${\cal S}_j$ in the
transmission conditions using a low frequency expansion in the Fourier
variable $k$ of the optimal choice given in Theorem
\ref{OptimalSchwarzThm}. This leads to the so called Taylor
transmission conditions (TTC), which have the symbols
\begin{equation}\label{eq:lowfreq}
  \begin{array}{rcl}
    \widehat{\cal S}_1(1,1)&=&\mathrm{i}\rho\omega C_p+\mathrm{i} \rho \frac{C_p^2}{2\omega}(C_p-2C_s)k^2+\mathcal{O}(k^4),\\
    \widehat{\cal S}_1(1,2)&=&-\mathrm{i}\rho(C_p-2C_s)C_sk+\mathcal{O}(k^3),\\
    \widehat{\cal S}_1(2,1)&=&\mathrm{i}\rho(C_p-2C_s)C_sk+\mathcal{O}(k^3),\\
    \widehat{\cal S}_1(2,2)&=&\mathrm{i}\rho\omega C_s+\mathrm{i}\rho \frac{C_s^2}{2\omega}(C_s-2C_p)k^2 + \mathcal{O}(k^4),
  \end{array}
\end{equation}
and $\widehat{\cal S}_2$ with the same relation to $\widehat{\cal S}_1$ as for the optimal choice in Theorem \ref{OptimalSchwarzThm}. A zeroth order approximation would thus be
\begin{equation}
  \widehat{\cal S}_1^{T_0}(1,1)=\mathrm{i}\rho\omega C_p,\quad
  \widehat{\cal S}_1^{T_0}(1,2)=0, \quad
  \widehat{\cal S}_1^{T_0}(2,1)=0, \quad
  \widehat{\cal S}_1^{T_0}(2,2)=\mathrm{i}\rho\omega C_s,
\end{equation}
which was also obtained as an ABC using a different argument in
\cite{Huttunen:04:UWV}. These ABCs happen to be exact for a particular
combination of plane waves, and thus have a physical sense for this
particular problem.

We show in Figure \ref{Fig3b} the modulus of the convergence factor
of the optimized Schwarz method with Taylor transmission conditions.
\begin{figure}
  \centering  
  \includegraphics[width=0.5\textwidth]{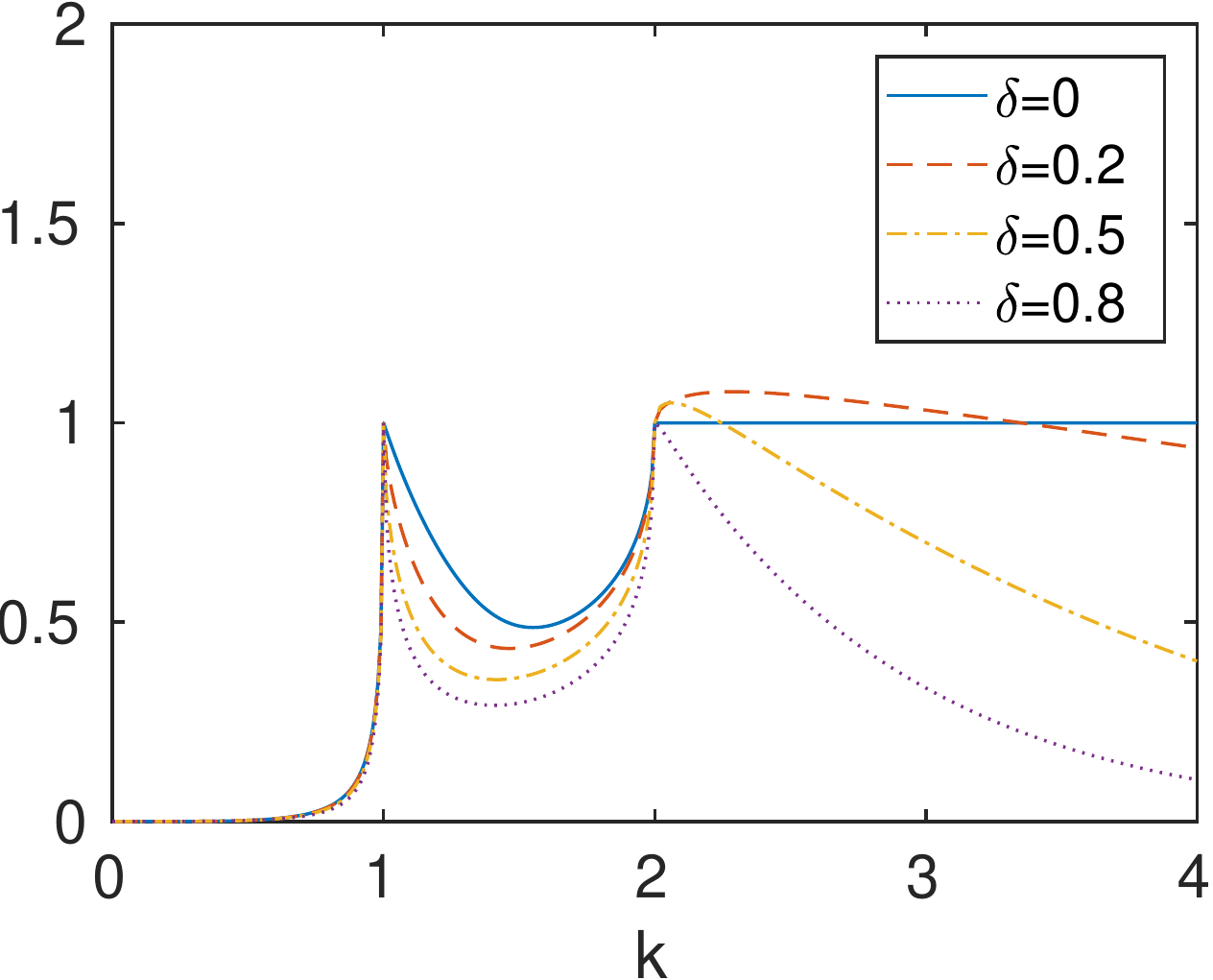}
  \caption{Modulus of the convergence factor of the optimized Schwarz
    method with Taylor transmission conditions for $C_p=1$,
    $C_s=\frac{1}{2}$, $\omega=1$ for different values of the overlap
    $\delta$.}
  \label{Fig3b}
\end{figure}
We see that the method now converges very well for low frequencies,
and also for intermediate frequencies. For high frequencies, we see
that without overlap, $\delta=0$, the method stagnates, since the
convergence factor equals 1. Increasing the overlap leads to
convergence of the very high frequencies, and when the overlap becomes
big enough, the method seems to converge for all frequencies, except at the
two points $k\in\{\frac{\omega}{C_p},\frac{\omega}{C_s}\}$. This is a
very important improvement compared to the classical Schwarz method,
see Figure \ref{Fig3a}, and while for Helmholtz equations there is one
non-convergent frequency when using optimized transmission conditions
\cite{Gander:2001:OSH,gander2007optimized,gander2016optimized}, for
the Navier equations there are two. We prove in the following theorem
that the numerical observations in Figure \ref{Fig3b} indeed hold for
all parameter choices in the Navier equations in the non-overlapping
case.
\begin{theorem}[Convergence of the non-overlapping Schwarz algorithm with TTC]
  The new Schwarz method \eqref{GeneralOptimizedSchwarz} with TTC
  \eqref{eq:lowfreq} for non-overlapping decompositions converges for
  $k \in (0,\frac{\omega}{C_s})\backslash \{\frac{\omega}{C_p}\} $,
  and stagnates with the contraction factor being equal to $1$ for $k
  \in [\frac{\omega}{C_s},\infty)$.
\label{OptTTCs}
\end{theorem}
\begin{proof}
The proof is again quite technical: the eigenvalues of the iteration
matrix are given by
\begin{equation}\label{eq:rpm}
 r_\pm = \frac{X^2}{2}+ Y \pm \frac{1}{2} \sqrt{X^2(X^2+4Y)},\quad  X = b_{11} - b_{22}, \quad Y = b_{11}b_{22}-b_{12}b_{21},
\end{equation}
where the elements in the matrix $B$ are given by
\begin{equation}
  \arraycolsep0.3em
B=\begin{bmatrix}
b_{11}  & b_{12} \\
b_{21}  & b_{22}
\end{bmatrix}:=\frac{1}{D} \begin{bmatrix}
-Z_1-Z_2 - \mathrm{i} \omega^3 ({\lambda}_{1}-{\lambda}_{2}\frac{C_p}{C_s}) & \mathrm{i} {\lambda}_{2} K \\
-\mathrm{i} {\lambda}_{1} K  & -Z_1-Z_2 +\mathrm{i}\omega^3({\lambda}_{1}-{\lambda}_{2}\frac{C_p}{C_s})
\end{bmatrix}, 
\label{BLke0}
\end{equation}
and $Z_1$, $Z_2$, $K$ and $D$ are defined by 
\begin{equation}\label{eq:kd}
  \arraycolsep0.2em
  \begin{array}{rclrcl}
Z_1&:=& C_s^3\left(k^2+\lambda_1^2\right)^2+\omega^2 C_p k^2,&
Z_2&:=& \left(4 C_s^3 k^2+ C_p\omega^2\right) \lambda_1 \lambda_2,\\
K&:=&2k\left(C_p\omega^2+2C_s^3\left(k^2+{\lambda}_{1}^2\right)\right),&
D&:=&-Z_1 + Z_2 +\mathrm{i}\omega^3 ({\lambda}_{1}+{\lambda}_{2}\frac{C_p}{C_s}).
  \end{array}
\end{equation}
We define now $\bar{\lambda}_j \in \mathbb{R_+},\, j=1,2$ as in
(\ref{ImagLambda}), and study the five cases for $k$ as in the proof
of Theorem \ref{ClassicalSchwarzTheorem}: if $k
\in(0,\frac{\omega}{C_p})$ then $\lambda_{1,2} \in
\mathrm{i} \mathbb{R_+}$, and using \eqref{eq:rpm} we obtain
$$
X = \frac{2\omega^3 }{{D}} \left(\bar{\lambda}_{1}-\bar{\lambda}_{2}\frac{C_p}{C_s}\right),\quad Y=\frac{1}{{D}^2} \left( ({Z}_1+{Z}_2)^2-\omega^6 \left( \bar{\lambda}_{1}-\bar{\lambda}_{2}\frac{C_p}{C_s}\right)^2 + \bar{\lambda}_{1} \bar{\lambda}_{2}{K}^2  \right).
$$
A direct computation shows that $X^2+2Y>0$ and $X^2+4Y>0$, and hence
$r_{+} >|r_{-}|>0$, so we just need to check \MG{that
$$
  r_{+}<1 \quad \Longleftrightarrow \quad
  \left(X^2+2Y\right) + \sqrt{X^2\left(X^2+4Y\right)} < 2.
$$
To show this second inequality, we compute 
$$
\begin{array}{l}
\left(X^2+2Y\right) + \sqrt{X^2\left(X^2+4Y\right)} < 2\\
\qquad  \Longleftrightarrow X^2\left(X^2+4Y\right) < \left(2(1-Y)-X^2\right)^2\\
\qquad \Longleftrightarrow X^4+4X^2 Y < 4(1-Y)^2-4(1-Y)X^2+X^4\\
  \qquad \Longleftrightarrow \left(1-Y\right)^2-X^2>0,
\end{array}
$$
and the last} inequality can be checked by first setting $X=
\widetilde{X}/{{D}}$ and $Y = \widetilde{Y}/{{D}}^2$, which leads to
the condition
$$
  0<(1- \widetilde{Y}/{D}^2)^2 - ( \widetilde{X}/{D})^2
  \quad\Longleftrightarrow\quad 0<({D}^2-\widetilde{Y})^2 -{D}^2\widetilde{X}^2 = 16\omega^6\frac{C_p}{C_s}\bar{\lambda}_1\bar{\lambda}_2 C^2,
$$
where $C \in\mathbb{R^*}$ is a complicated factor depending on $C_p$,
$C_s$, $\omega$, and $k$, and the other terms are positive.  We thus
conclude that in this case the algorithm is convergent.

If $k =\frac{\omega}{C_p}$ then $\lambda_{1} =
\mathrm{i}\frac{\omega\sqrt{C_p^2-C_s^2}}{C_s C_p}$ and
$\lambda_{2}=0$, and the elements in the matrix $B$ are
$$
b_{11}=
\frac{(C_p+C_s)(C_p^3-4C_pC_s^2+4C_s^3)-\sqrt{C_p^2-C_s^2} C_p^3}{(C_p+C_s)(C_p^3-4C_pC_s^2+4C_s^3) + \sqrt{C_p^2-C_s^2} C_p^3},\quad b_{12}=0,\quad b_{21}\in\mathbb{C}, \quad b_{22}=1,
$$
and the eigenvalues $r_{\pm}$ are given by
$$
r_+ = 1,\quad 
|r_-| = \left|\frac{\left(C_p+C_s\right)\left(C_p^3-4C_pC_s^2+4C_s^3\right) - \bar{\lambda}_1 C_p^4C_s}{\left(C_p+C_s\right)\left(C_p^3-4C_pC_s^2+4C_s^3\right) + \bar{\lambda}_1 C_p^4C_s}\right|^2.
$$
Since $C_p^3-4C_pC_s^2+4C_s^3>0$, we have $|r_-| < 1$, and thus $\rho_{T_0} = 1$.

If $k\in(\frac{\omega}{C_p},\frac{\omega}{C_s})$ then $\lambda_{1} \in
\mathrm{i}\mathbb{R_+}$ and $\lambda_{2} \in \mathbb{R_+}$, and we
obtain
$$
r_{\pm}=\left(\dfrac{\omega^3\left(\bar{\lambda}_1+\mathrm{i}\lambda_2\frac{C_p}{C_s}\right)\pm\sqrt{-\mathrm{i}\lambda_2\bar{\lambda}_1{K}^2-\left(\bar{Z_2}-\mathrm{i}{Z_1}\right)^2}}{\left(-{Z_1}+ \mathrm{i}\bar{Z_2}\right)-\omega^3\left(\bar{\lambda}_1-\mathrm{i}\lambda_2\frac{C_p}{C_s}\right)}\right)^2.
$$
By computing their modulus, we get
$$
\begin{aligned}
|r_{\pm}| &= \textstyle \frac{\left(\omega^3\frac{C_p}{C_s}\lambda_2 \mp \csgn(\alpha) \frac{\sqrt{2}}{2} \sqrt{ \sqrt{ \left({Z}_1^2-\bar{Z}_2^2\right)^2 + \left(K^2\lambda_2\bar{\lambda}_1-2{Z}_1\bar{Z}_2\right)^2} -{Z}_1^2 + \bar{Z}_2^2}\right)^2}{\left(\omega^3\frac{C_p}{C_s}\lambda_2 + \bar{Z_2}\right)^2+\left(\omega^3\bar{\lambda}_{1}+{Z_1}\right)^2} \\
&+ \textstyle\frac{\left(\omega^3\bar{\lambda}_1 \pm \frac{\sqrt{2}}{2} \sqrt{ \sqrt{ \left({Z}_1^2-\bar{Z}_2^2\right)^2 + \left(K^2 \lambda_2\bar{\lambda}_1-2{Z}_1\bar{Z}_2\right)^2}+{Z}_1^2-\bar{Z}_2^2}\right)^2}{\left(\omega^3\frac{C_p}{C_s}\lambda_2 + \bar{Z_2}\right)^2+\left(\omega^3\bar{\lambda}_{1}+{Z_1}\right)^2},
\end{aligned}
$$
where 
$$
\begin{aligned}
\alpha = \left(K^2\lambda_2\bar{\lambda}_1-2{Z}_1\bar{Z}_2+\mathrm{i}\left({Z}_1^2-\bar{Z}_2^2\right)\right), \quad \bar{Z}_2 = \left(4 C_s^3 k^2+ C_p\omega^2\right) \bar{\lambda}_1 \bar{\lambda}_2.
\end{aligned}
$$
An upper bound $M$ for the modulus of the eigenvalues is thus
obtained choosing the plus sign,
$$
\begin{aligned}
M :=  &\textstyle\frac{\left(\omega^3\frac{C_p}{C_s}\lambda_2+\frac{\sqrt{2}}{2} \sqrt{ \left( \left({Z}_1^2-\bar{Z}_2^2\right)^2 + \left(K^2\lambda_2\bar{\lambda}_1-2{Z}_1\bar{Z}_2\right)^2\right)^{\frac{1}{2}}-{Z}_1^2+\bar{Z}_2^2}\right)^2}{\left(\omega^3\frac{C_p}{C_s}\lambda_2+\bar{Z_2}\right)^2+\left(\omega^3\bar{\lambda}_{1}+{Z_1}\right)^2} \\
+ & \textstyle\frac{\left(\omega^3\bar{\lambda}_1+\frac{\sqrt{2}}{2} \sqrt{ \left( \left({Z}_1^2-\bar{Z}_2^2\right)^2+\left(K^2\lambda_2\bar{\lambda}_1-2{Z}_1\bar{Z}_2\right)^2\right)^{\frac{1}{2}}+{Z}_1^2-\bar{Z}_2^2}\right)^2}{\left(\omega^3\frac{C_p}{C_s}\lambda_2+\bar{Z_2}\right)^2+\left(\omega^3\bar{\lambda}_{1}+{Z_1}\right)^2},
\end{aligned}
$$
and it suffices to prove that $M<1$. \MG{To do so, it is sufficient
  to show that for the numerator in the first term
of $M$, we have}
\begin{equation}
\label{eq:M1}
0 < \omega^3\frac{C_p}{C_s}\lambda_2+\frac{\sqrt{2}}{2} \sqrt{ \left( \left({Z}_1^2-\bar{Z}_2^2\right)^2 + \left(K^2\lambda_2\bar{\lambda}_1-2{Z}_1\bar{Z}_2\right)^2\right)^{\frac{1}{2}}-{Z}_1^2+\bar{Z}_2^2}
< \omega^3\frac{C_p}{C_s}\lambda_2+\bar{Z_2},
\end{equation}
and for the numerator in the second term of $M$, we \MG{have}
\begin{equation}
\label{eq:M2}
0 < \omega^3\bar{\lambda}_1+\frac{\sqrt{2}}{2} \sqrt{ \left( ({Z}_1^2-\bar{Z}_2^2)^2 + (K^2 \lambda_2\bar{\lambda}_1-2{Z}_1\bar{Z}_2)^2\right)^{\frac{1}{2}}+{Z}_1^2-\bar{Z}_2^2}
< \omega^3\bar{\lambda}_{1}+{Z_1}.
\end{equation}
\MG{By a direct computation, one can show that} both \eqref{eq:M1} and \eqref{eq:M2} are equivalent to
$$
0<4{Z_1}\bar{Z_2}-{K}^2\lambda_2\bar{\lambda}_1=4\lambda_2\bar{\lambda}_1\frac{C_p}{C_s}\omega^6,
$$
which \MG{clearly holds, and thus}
$\max\left(|r_+|,|r_-|\right)\leq M<1$ and the algorithm is convergent.

If $k =\frac{\omega}{C_s}$ then $\lambda_{1} = 0$ and
$\lambda_{2}=\frac{\omega\sqrt{C_p^2-C_s^2}}{C_s C_p}>0$.  In this
case the coefficients of the matrix $B$ are
given by
$$
b_{11}= 1,\quad b_{12}\in\mathbb{C},\quad b_{21}=0, \quad b_{22}=\frac{ - \mathrm{i}\sqrt{C_p^2-C_s^2} - (C_p+C_s)}{\mathrm{i} \sqrt{C_p^2-C_s^2} - (C_p+C_s)}\MG{,}
$$
and the eigenvalues $r_{\pm}$ are
$$
r_+ = 1,\quad 
|r_-| = \left|\frac{ - \mathrm{i}\sqrt{C_p^2-C_s^2} - (C_p+C_s)}{\mathrm{i} \sqrt{C_p^2-C_s^2} - (C_p+C_s)}\right|^2=1, 
$$
and the algorithm therefore stagnates for $k=\frac{\omega}{C_s}$.

If $k \in \left(\frac{\omega}{Cs},\infty\right)$ then $\lambda_{1,2}
\in \mathbb{R^*_+}$ and \eqref{eq:rpm} gives $r_\pm = \dfrac{1}{D}(R
\pm \mathrm{i}I)$ with
\begin{equation}\label{R1I1case5L0}
\begin{aligned}
R &= \textstyle-K^2\lambda_1\lambda_2-\omega^6 \left(\lambda_1-\lambda_2\frac{C_p}{C_s}\right)^2 + \left(Z_1+Z_2\right)^2, \\
I &= \textstyle-2\omega^3 \left(\lambda_1-\lambda_2\frac{C_p}{C_s}\right) \sqrt{\left(Z_1+Z_2\right)^2 - K^2\lambda_1\lambda_2}.
\end{aligned}
\end{equation}
A direct computation shows that \MG{for
  $C(\omega,k,C_p,C_s)\in\mathbb{R}^*$ a constant}
$$
  R^2+I^2-|D|^2 = C(\omega,k,C_p,C_s) \left(K^2\lambda_1 \lambda_2
   - 4\left(Z_1Z_2-\omega^3\lambda_1\lambda_2 \frac{C_p}{C_s}\right)\right) =0,
$$
\MG{since $
K^2\lambda_1 \lambda_2 - 4Z_1Z_2+4\omega^3\lambda_1\lambda_2 \frac{C_p}{C_s}= 8 \lambda_1 \lambda_2 C_s^3\left(4C_s^3k^2+C_p\omega^2\right)
\left(k^2+\lambda_1^2 \right)(k^2-\lambda_1^2-\frac{\omega^2}{C_s^2})$, and
$k^2-\lambda_1^2-\frac{\omega^2}{C_s^2}=0$},
and hence 
$|r_{\pm}| = 1$ and the algorithm stagnates.
\end{proof}

The non-overlapping Schwarz algorithm with Taylor transmission
conditions thus leads to good convergence for low frequencies, but
stagnates for high frequencies. We now investigate if the combination
of overlap and TTC can lead to a convergent optimized Schwarz
algorithm. A first result for strictly positive overlap $\delta>0$
is the following, see also Figure \ref{Fig3b} for an illustration:

\begin{theorem}[Convergence of the overlapping Schwarz algorithm with TTC.]
\label{OverlappingSchwarzTTCTheorem}
  For $\delta>0$ small, the new overlapping Schwarz method
\eqref{GeneralOptimizedSchwarz} with Taylor transmission conditions
\eqref{eq:lowfreq} converges for
$$
\textstyle k \in (0,\frac{\omega}{C_p}) \cup (\frac{\omega}{C_p},\frac{\omega}{C_s})
\cup (k^*,\infty),\qquad k^*(\omega,C_p,C_s,\delta) \in (\frac{\omega}{C_s},\infty),
$$
diverges for \MG{$k \in (\frac{\omega}{C_s},k^*)$, and stagnates
  for $k\in\{\frac{\omega}{C_p},\frac{\omega}{C_s},k^*\}$}.
\label{OptTTCsL}
\end{theorem}
\begin{proof}
Again the proof is quite technical: the eigenvalues of the iteration
matrix are
\begin{equation}
r_\pm = \frac{X^2}{2}+ Y \pm \frac{1}{2} \sqrt{X^2\left(X^2+4Y\right)},\quad X = \expo^{-\lambda_1 \delta}b_{11} - \expo^{-\lambda_2 \delta}b_{22}, \,
Y = \frac{b_{11}b_{22}-b_{12}b_{21}}{\expo^{\lambda_1 \delta}\expo^{\lambda_2 \delta}},
\label{eq:rpm2}
\end{equation}
where the elements of the matrix $B$ are
$$
B = \begin{bmatrix}
b_{11}  & b_{12} \\
b_{21}  & b_{22}
\end{bmatrix}=
\frac{1}{D} \begin{bmatrix}
-Z_1-Z_2 - \mathrm{i} \omega^3 \left({\lambda}_{1}-{\lambda}_{2}\frac{C_p}{C_s}\right) & \mathrm{i} {\lambda}_{2} K \\
-\mathrm{i} {\lambda}_{1} K  & -Z_1-Z_2 +\mathrm{i}\omega^3\left({\lambda}_{1}-{\lambda}_{2}\frac{C_p}{C_s}\right)
\end{bmatrix}
$$
and $Z_1$, $Z_2$, $K$ and $D$ are given by
$$
\begin{array}{c}
Z_1= C_s^3\left(k^2+\lambda_1^2\right)^2+\omega^2 C_p k^2,\quad
Z_2= \left(4 C_s^3 k^2+ C_p\omega^2\right) \lambda_1 \lambda_2,\\[2ex]
K=2k\left(C_p\omega^2+2C_s^3\left(k^2+{\lambda}_{1}^2\right)\right),\quad
D=-Z_1 + Z_2 +\mathrm{i}\omega^3 \left({\lambda}_{1}+{\lambda}_{2}\frac{C_p}{C_s}\right).
\end{array}
$$
We define $\bar{\lambda}_j \in \mathbb{R_+}$, $j=1,2$, as in
(\ref{ImagLambda}) when ${\lambda}_1$ and/or ${\lambda}_2 \in
i\mathbb{R}$. When the overlap $\delta$ is small, a
series expansion of the eigenvalues gives
\begin{equation}
{r_\pm} = \left(R_{1\pm}+\mathrm{i}I_{1\pm}\right) + \left(R_{2\pm}+\mathrm{i}I_{2\pm}\right)\delta + \mathcal{O}(\delta^2),\quad (R_{j\pm},I_{j\pm})\in \mathbb{R},
\label{series}
\end{equation}
and the modulus of the eigenvalues becomes
$$
|{r_\pm}|^2 = \left(R_{1\pm}^2 + I_{1\pm}^2\right) + 2\delta\left(R_{1\pm}R_{2\pm}+I_{1\pm}I_{2\pm}\right) + \mathcal{O}(\delta^2).
$$
Again we need to distinguish several cases: if $k
\in(0,\frac{\omega}{C_p})$ then $\lambda_{1,2} \in
\mathrm{i} \mathbb{R_+}$ and $I_{1\pm}=R_{2\pm}=0$ for both
eigenvalues.  Therefore the series expansion (\ref{series}) becomes
$$
  {r_{\pm}} = R_{1\pm}+\mathrm{i}I_{2\pm}\delta+\mathcal{O}(\delta^2) \quad
    \Longrightarrow\quad |{r_{\pm}}|^2 = R_{1\pm}^2 + \mathcal{O}(\delta^2), 
$$
where 
$$
\begin{aligned}
 R_{1\pm}=\;& \textstyle\frac{\omega^6\left(\bar{\lambda}_1-\bar{\lambda}_2 \frac{C_p}{C_s}\right)^2 + ({Z}_1+{Z}_2)^2+4 k^2 \bar{\lambda}_1\bar{\lambda}_2 \left(4C_s^3 k^2+C_p\omega^2-2C_s\omega^2\right)^2}{\left({Z}_1-{Z}_2+\omega^3\left(\bar{\lambda}_1+\bar{\lambda}_2 \frac{C_p}{C_s}\right)\right)^2}\\
 &\textstyle\pm 2\omega^3 \left(\bar{\lambda}_1-\bar{\lambda}_2 \frac{C_p}{C_s}\right)\frac{\sqrt{\left({Z}_1+{Z}_2\right)^2+4 k^2 \bar{\lambda}_1\bar{\lambda}_2 \left(4C_s^3 k^2+C_p\omega^2-2C_s\omega^2\right)^2}}{\left({Z}_1-{Z}_2+\omega^3\left(\bar{\lambda}_1+\bar{\lambda}_2 \frac{C_p}{C_s}\right)\right)^2}.
\end{aligned}
$$
After simplifications, this gives exactly the same convergence factor
as in the non-overlapping case for which we have proved in Theorem
\ref{OptTTCs} that it is less
than one. Therefore the algorithm is convergent in this case for
$\delta>0$ small enough.

If $k =\frac{\omega}{C_p}$ then $\lambda_{1} =
\mathrm{i}\frac{\omega\sqrt{C_p^2-C_s^2}}{C_s C_p}$ and
$\lambda_{2}=0$. In this case the elements of the matrix $B$ are
$$
b_{11}=\textstyle\frac{(C_p+C_s)\left(C_p^3-4C_pC_s^2+4C_s^3\right)-\sqrt{C_p^2-C_s^2} C_p^3}{(C_p+C_s)\left(C_p^3-4C_pC_s^2+4C_s^3\right)+\sqrt{C_p^2-C_s^2}C_p^3},\ b_{12}=0,\ b_{21}\in\mathbb{C}, \ b_{22}=1,
$$
and the eigenvalues $r_{\pm}$ are
$$
r_+ = 1,\quad |r_-| = \left|\expo^{-2\mathrm{i}\bar{\lambda}_1 \delta}\frac{(C_p+C_s)(C_p^3-4C_pC_s^2+4C_s^3) - \bar{\lambda}_1 C_p^4C_s}{(C_p+C_s)(C_p^3-4C_pC_s^2+4C_s^3) + \bar{\lambda}_1 C_p^4C_s}\right|^2.
$$
Since $C_p^3-4C_pC_s^2+4C_s^3>0$, we have $|r_-| < 1$, and thus
$\rho_{T_0} = 1$ which means the algorithm stagnates in this case.

If $k \in (\frac{\omega}{C_p},\frac{\omega}{C_s})$, then
$\lambda_{1} \in \mathrm{i}\mathbb{R_+}$ and $\lambda_{2} \in
\mathbb{R_+}$.  The series expansion (\ref{series}) becomes
$$
|{r_\pm}|^2 = \left(R_{1\pm}^2 + I_{1\pm}^2\right) + \mathcal{O}(\delta), 
$$
and the terms $(R_{1\pm}+\mathrm{i}I_{1\pm})$ are the same as in the
non-overlapping case, and we already know from the proof
of Theorem \ref{OptTTCs} that $\left(R_{1\pm}^2 +
I_{1\pm}^2\right)<1$. Therefore the algorithm is convergent in this
case for $\delta>0$ small enough\footnote{From Figure \ref{Fig3b} we
  see that actually the overlap makes the algorithm faster in this
  interval, and even slightly faster also in the first interval.}.

If $k =\frac{\omega}{C_s}$ then 
$
\lambda_{1} = 0,\,\lambda_{2}=\frac{\omega\sqrt{C_p^2-C_s^2}}{C_s C_p}>0.
$
In this case the elements in the matrix $B$ are 
$$
  \textstyle b_{11}= 1,\quad b_{12}\in\mathbb{C},\quad b_{21}=0, \quad b_{22}=  \frac{ - \mathrm{i}\sqrt{C_p^2-C_s^2} - (C_p+C_s)}{\mathrm{i} \sqrt{C_p^2-C_s^2} - (C_p+C_s)},
$$
and the eigenvalues $r_{\pm}$ of the iteration matrix are given by
$$
r_+ = 1,\quad 
|r_-| = \expo^{-2\lambda_2\delta}\left|\frac{-\mathrm{i}\sqrt{C_p^2-C_s^2}-(C_p+C_s)}{\mathrm{i} \sqrt{C_p^2-C_s^2}-(C_p+C_s)}\right|^2=\expo^{-2\lambda_2\delta} < 1,
$$
which shows that the algorithm stagnates.

Finally, if $k \in (\frac{\omega}{C_s},\infty)$, then
$\lambda_{1,2}\in\mathbb{R^*_+}$ and the eigenvalues are given by
\eqref{eq:rpm2}.  We then use for $\delta>0$ small the series
expansion (\ref{series}) for $r_\pm$ and obtain
$$
R_{1\pm}+\mathrm{i}I_{1\pm}=\frac{1}{D}(R\pm\mathrm{i}I)\MG{,}
$$ 
where the values ($R,I,D$) are given in (\ref{R1I1case5L0}) and (\ref{eq:kd}).
Hence $R_{1\pm}^2+I_{1\pm}^2=1$ and
$$
\begin{aligned}
R_{1+}R_{2+}&+I_{1+}I_{2+}=-\frac{\left({\lambda}_{2}\lambda_1 K^2 - \omega^6\left(\lambda_1-\lambda_2\frac{C_p}{C_s}\right)^2 - \left(Z_1+Z_2\right)^2\right)^2}{|D|\sqrt{(Z_1+Z_2)^2 - {\lambda}_{2}\lambda_1 K^2}}\\
&\times \left(\sqrt{(Z_1+Z_2)^2-{\lambda}_{2}\lambda_1 K^2}(\lambda_1+\lambda_2) - (\lambda_1-\lambda_2)(Z_1+Z_2)\right) < 0, 
\end{aligned}
$$
since $(\lambda_1-\lambda_2)<0<(Z_1+Z_2)$. As the first eigenvalue is
less than one\MG{,}
$$
  r_+\sim 1+2\left(R_{1+}R_{2+}+I_{1+}I_{2+}\right)\delta <1,
$$
we will focus now on $r_-\sim 1+R_{1-}R_{2-}+I_{1-}I_{2-}=:F(k)$, with 
\begin{equation}
\begin{aligned}
F(k)&=-\frac{\left({\lambda}_{2}\lambda_1 K^2 - \omega^6\left(\lambda_1-\lambda_2\frac{C_p}{C_s}\right)^2 - (Z_1+Z_2)^2\right)^2}{|D|\sqrt{(Z_1+Z_2)^2 - {\lambda}_{2}\lambda_1 K^2}}\\
&\times \left(\underbrace{\sqrt{(Z_1+Z_2)^2 - {\lambda}_{2}\lambda_1 K^2}(\lambda_1+\lambda_2) + (\lambda_1-\lambda_2)(Z_1+Z_2)}_{g(k)}\right).
\label{eq:rm}
\end{aligned}
\end{equation}
Note that $\sqrt{(Z_1+Z_2)^2 - {\lambda}_{2}\lambda_1 K^2}\in\mathbb{R}$ since we have
\begin{equation}
\begin{aligned}
&(Z_1+Z_2)^2 - \lambda_{2}\lambda_1 K^2>0\quad \Longleftrightarrow\quad (Z_1+Z_2) - \sqrt{\lambda_{2}\lambda_1} K>0\\ 
&\Longleftrightarrow\quad (4C_s^3k^2+C_p\omega^2) t^2 - K t+Z_1>0, \quad t=\sqrt{\lambda_{2}\lambda_1},
\end{aligned}
\label{remark1}
\end{equation}
which \MG{holds because the discriminant}
$K^2-4(4C_s^3k^2+C_p\omega^2)Z_1=-\frac{4C_p\omega^6}{C_s}<0$.  So we
\MG{do not} have real solutions and the dominant term being positive, we
conclude this inequality holds \MG{for all} $k > \frac{\omega}{C_s}$.  We
can conclude that $g(k)\in\mathbb{R}$ as we have seen previously. We
now need to investigate under which conditions $g(k)<0$ which is
equivalent to $r_- >1$. \MG{By a direct calculation, we obtain}
$$
\begin{aligned}
g(k)<0\quad &\Longleftrightarrow\quad \sqrt{(Z_1+Z_2)^2 - {\lambda}_{2}\lambda_1 K^2}(\lambda_1+\lambda_2) < (\lambda_2-\lambda_1)(Z_1+Z_2)\\
&\Longleftrightarrow\quad \left((Z_1+Z_2)^2 - {\lambda}_{2}\lambda_1 K^2\right)(\lambda_1+\lambda_2)^2 < (\lambda_2-\lambda_1)^2(Z_1+Z_2)^2\\
&\Longleftrightarrow\quad \underbrace{2(Z_1+Z_2)-K(\lambda_1+\lambda_2)}_{\VD{\tilde g}}<0.
\end{aligned}
$$
We next study the sign of \VD{$\tilde g$} in a neighborhood of
$\frac{\omega}{C_s}$: we set
$k=\frac{\omega}{C_s}+\varepsilon$, and then expand \VD{$\tilde g$} in a series for
$\varepsilon$ small, which leads to
$$
\begin{aligned}
\VD{\tilde g} = &\textstyle\frac{2\omega^4}{C_pC_s^2} \left(\left(C_s+C_p\right)C_p-\left(C_p+2C_s\right)\sqrt{C_p^2-C_s^2}\right) \\
&- \textstyle\frac{2}{C_p}\sqrt{\frac{2\omega^7}{C_s^3}} \left(C_p\left(C_p+2C_s\right)-\left(C_p+4C_s\right)\sqrt{C_p^2-C_s^2}\right)\sqrt{\varepsilon} + \mathcal{O} (\varepsilon).
\end{aligned}
$$
For sufficiently small values of $\varepsilon$, that is for $k$ close
to $\frac{\omega}{C_s}$, the leading term of this series being
negative, we have $r_->1$ for $\delta>0$ small enough and the
algorithm diverges.  On the other hand, because of the overlap,
$\lim\limits_{k\to\infty}\rho_{T_0}(k,\omega,C_p,C_s,\delta)=0$ and by
continuity there exist two values $k^*>\bar{k}>\frac{\omega}{C_s}$
such that for all $k>k^*$ we have ${\rho_{T_0}}(k)<1$,
at $k=k^*$ we have $\rho_{T_0}(k^*,\omega,C_p,C_s,\delta)=1$, and
$|\rho_{T_0}(\bar{k})|>1$ with
$\bar{k}=\mbox{argmax}_{k>\frac{\omega}{C_s}}|\rho_{T_0}|\in(\frac{\omega}{C_s},k^*)$, which
concludes the proof for small overlap $\delta>0$.
\end{proof}

It is possible to obtain an asymptotic estimate for $k^*$ and also the
rate at which the method diverges for the frequency $k=\bar{k}$, see
the PhD thesis \cite[pp. 45 ff]{BrunetPhD2018}. We focus however next
on how to obtain a convergent algorithm. The results of Theorem
\ref{OptTTCsL} hold for overlap $\delta>0$ small enough: if the
overlap is bigger, it is possible to obtain a convergent optimized
Schwarz method except for the two isolated frequencies
$k=\frac{\omega}{C_p}$ and $k=\frac{\omega}{C_s}$, as indicated in
Figure \ref{Fig3b} for $\delta=0.8$, where the bump in the convergence
factor making it larger than one has disappeared. In the Helmholtz
case, there is also one isolated frequency which is not convergent
when using an optimized Schwarz method
\cite{Gander:2001:OSH,gander2007optimized,gander2016optimized}, and
such isolated cases can be left to Krylov acceleration. We are
therefore interested in estimating the value
$\delta^*(C_p,C_s,\omega)$ for which the optimized Schwarz method with
Taylor transmission conditions converges as soon as the overlap
$\delta>\delta^*(C_p,C_s,\omega)$ like illustrated in Figure
\ref{Fig5},
\begin{figure}
  \centering
  \includegraphics[width=0.3\textwidth]{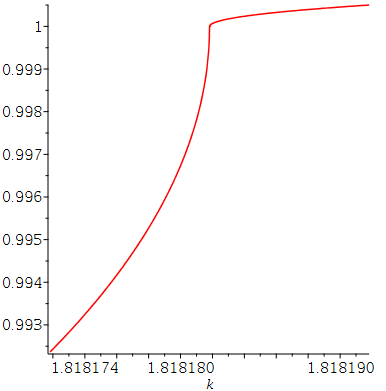}
  \includegraphics[width=0.3\textwidth]{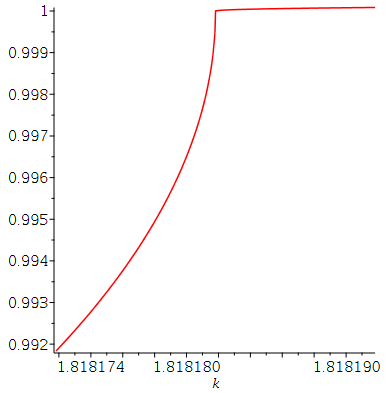}
  \includegraphics[width=0.3\textwidth]{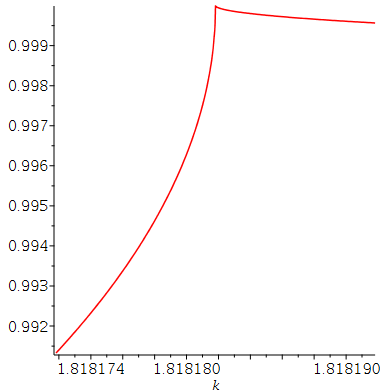}
  \caption{Convergence factor $\rho_{T_0}$ close to
    $k=\frac{\omega}{C_s}$ for $C_p=1$, $C_s=\frac12$ and $\omega=1$.
    Left: $\delta=0.8$ (divergence). Middle: $\delta=0.9$
    (approximate stagnation). Right: $\delta=1$ (convergence).}
  \label{Fig5}
\end{figure}
where we see with a zoom that $\delta=0.8$ is not quite enough for
convergence, but $\delta=1$ is.
\begin{theorem}
  The new overlapping Schwarz algorithm \eqref{GeneralOptimizedSchwarz}
  with Taylor transmission conditions \eqref{eq:lowfreq} converges for
  $k\in\mathbb{R}_+ \:\backslash\:
  \left\{\frac{\omega}{C_p},\frac{\omega}{C_s}\right\}$ if the overlap
  $\delta$ is bigger than
  $$
    {\delta}^*(C_p,C_s,\omega) =
      \frac{C_s\sqrt{C_p^2-C_s^2}(C_p+2C_s)^2}{C_p\omega(C_s+C_p)}
      \frac{\sinh(\alpha)}{C_p\cosh(\alpha) +C_s},
  $$
  where $\alpha$ is the positive root of
  $$
  \alpha C_p^2\left(C_p\cosh(\alpha)+C_s\right) -
  \left(C_p^3+\VD{(\alpha-1)(}3C_p^2C_s-4C_s^3\VD{)}\right)\sinh(\alpha)= 0.
  $$
\end{theorem}
\begin{proof}
As illustrated in Figure \ref{Fig5}, we need to investigate how the
convergent algorithm turns into a divergent one when $\delta$ is
decreased. \MG{For $k \in(0,\frac{\omega}{C_s})\backslash
  \{\frac{\omega}{C_p}\}$, the Schwarz algorithm with absorbing
  boundary conditions converges both without overlap (see Theorem
  \ref{OptTTCs}) and with a small overlap (see Theorem
  \ref{OptTTCsL}), and a bigger overlap only improves the behavior,
  so divergence does not happen for those values of $k$. If
  $k\in\{\frac{\omega}{C_p},\frac{\omega}{C_s}\}$, we know that the
  convergence factor is independent of the size $\delta$ of the
  overlap and always equals $1$, so the algorithm stagnates there.
  Only if $k \in(\frac{\omega}{C_s},\infty)$, the algorithm could
  diverge, and we thus} need to study the slope of the eigenvalues of
the iteration matrix \MG{at} $\frac{\omega}{Cs}$ coming from the
right\MG{, see Figure \ref{Fig5}. To do so,} we set
$k:=\frac{\omega}{C_s}+\varepsilon$ for $\varepsilon$ a parameter and
expand $r_\pm$ in a series as in (\ref{series}) for $\varepsilon$
small, with $\left(R_{j\pm},I_{j\pm}\right)\in \mathbb{R}$, and then
obtain for the modulus of the eigenvalues
$$
\begin{aligned}
|{r_\pm}|^2 &= \left(R_{1\pm}^2 + I_{1\pm}^2\right) + 2\sqrt{\varepsilon}\left(R_{1\pm}R_{2\pm}+I_{1\pm}I_{2\pm}\right) + \mathcal{O}(\varepsilon).
\end{aligned}
$$
For $r_+$, we obtain for the first term that
$$
  \textstyle R_{1+}+\mathrm{i}I_{1+} =-\frac{C_p^2-2C_s^2  - \mathrm{i}2C_s\sqrt{C_p^2-C_s^2}}{C_p^2\expo^{\frac{2\omega\sqrt{C_p^2-C_s^2}}{C_pC_s}\delta}} \quad\Longrightarrow\quad  R_{1+}^2 + I_{1+}^2 = \expo^{-\frac{4\omega\sqrt{C_p^2-C_s^2}}{C_sC_p}\delta}<1,
$$
and similarly for $r_-$ we get $R_{1-}+\mathrm{i}I_{1-}=1
\;\Longrightarrow\; R_{1-}^2 + I_{1-}^2=1$. For the second term, we get
$$
  \textstyle R_{1+}R_{2+}+I_{1+}I_{2+} = -\frac{ 2\sqrt{2C_s} \expo^{-\frac{4\omega\sqrt{C_p^2-C_s^2}}{C_sC_p}\delta} \sqrt{C_p^2-C_s^2}\left(C_p+2C_s\right)^2  \left(\expo^{\frac{2\omega\sqrt{C_p^2-C_s^2}}{C_sC_p}\delta}-1\right) }{C_p\sqrt{\omega}(C_p+C_s) \left(C_p\expo^{\frac{2\omega\sqrt{C_p^2-C_s^2}}{C_sC_p}\delta} + 2C_s\expo^{\frac{\omega\sqrt{C_p^2-C_s^2}}{C_sC_p}\delta} +C_p \right)}<0,
$$
from which we can conclude that $|{r_+}|^2 < 1$. For the second
eigenvalue, we get however
$$
\begin{array}{l}
R_{1-}R_{2-}+I_{1-}I_{2-} = \\
-\frac{2\sqrt{2}(\omega C_s)^{-\frac{1}{2}}}{(C_p+C_s)C_p} \underbrace{\left(\textstyle \delta C_p\omega(C_p+C_s) - \frac{C_s(C_p+2C_s)^2\sqrt{C_p^2-C_s^2} \left(\expo^{\frac{2\omega\sqrt{C_p^2-C_s^2}}{C_sC_p}\delta}-1\right)}{ C_p\expo^{\frac{2\omega\sqrt{C_p^2-C_s^2}}{C_sC_p}\delta} + 2C_s\expo^{\frac{\omega\sqrt{C_p^2-C_s^2}}{C_sC_p}\delta} +C_p}\right)}_{=:f(\delta)}.\\
\end{array}
$$
We therefore need to study the function $f$ to investigate for which
values of $\delta$ it is becoming negative, which means the algorithm
will diverge. Computing the derivative, we obtain
\begin{equation}
f'(\delta) = -\frac{2\sqrt{2\omega}\expo^{\frac{2\omega\sqrt{C_p^2-C_s^2}}{C_sC_p}\delta}}{\sqrt{C_s}C_p^2 \left( C_p \expo^{\frac{2\omega\sqrt{C_p^2-C_s^2}}{C_sC_p}\delta} + 2C_s\expo^{\frac{\omega\sqrt{C_p^2-C_s^2}}{C_sC_p}\delta} +C_p \right)^2}g(\delta),
\end{equation}
so the sign of $f'$ is the opposite sign of $g$ given by
$$
\begin{aligned}
g(\delta) = \;&2C_p^4 \cosh\left( \frac{2\omega}{C_sC_p}\sqrt{C_p^2-C_s^2}\delta\right) - 2C_p(C_p^3+6C_p^2C_s-2C_pC_s^2-8C_s^3) \\
&+ 4C_s(C_p+C_s)(C_p-2C_s)^2 \cosh\left( \frac{\omega}{C_sC_p}\sqrt{C_p^2-C_s^2}\delta\right).
\end{aligned}
$$
Computing the derivative of $g$, we find
$$
\begin{aligned}
 g'(\delta) = \;&\frac{4}{C_p} \omega (C_p+C_s)(C_p-2C_s)^2 \sqrt{C_p^2-C_s^2}\sinh\left( \frac{\omega}{C_sC_p}\sqrt{C_p^2-C_s^2}\delta\right) \\
 &+ \frac{4}{C_s} \omega \sqrt{C_p^2-C_s^2} \sinh\left( \frac{2\omega}{C_sC_p}\sqrt{C_p^2-C_s^2}\delta\right),
\end{aligned}
$$ and we have $g(0) = -8C_s(C_p+C_s)(C_p^2-2C_s^2)$. This shows that
$g'(\delta)>0$ for all $\delta>0$ and $g(0)<0$, since $C_p$ and $C_s$
are positive and $C_p^2>2C_s^2$, see \eqref{lambda12}.  Now $\cosh$ is
a strictly increasing function for positive arguments, and in our case
all the parameters are real and positive, and for $\delta=0$ we have
$\cosh(0)=1$. We therefore have
$$
\cosh\left(\frac{\omega}{C_sC_p}\sqrt{C_p^2-C_s^2} {\delta}\right)\leq\cosh\left(2\frac{\omega}{C_sC_p}\sqrt{C_p^2-C_s^2} {\delta}\right),
$$
and can thus estimate $g$ from below, 
\begin{equation}\label{LowBnd}
\begin{aligned}
g({\delta}) \geq \; &2C_p^4 \cosh\left( \frac{\omega}{C_sC_p}\sqrt{C_p^2-C_s^2}\delta\right) - 2C_p(C_p^3+6C_p^2C_s-2C_pC_s^2-8C_s^3) \\
 &+ 4C_s(C_p+C_s)(C_p-2C_s)^2 \cosh\left( \frac{\omega}{C_sC_p}\sqrt{C_p^2-C_s^2}\delta\right).
\end{aligned}
\end{equation}
Let $\bar{\delta}\in\mathbb{R}_+^*$ be the unique value of $\delta$
such that $\cosh\left(\frac{\omega}{C_sC_p}\sqrt{C_p^2-C_s^2}
\bar{\delta}\right)=3$; then we obtain from \eqref{LowBnd} the lower
bound
$$
\begin{aligned}
g(\bar{\delta}) &\geq 2C_p^4 \times 3 - 2C_p(C_p^3+6C_p^2C_s-2C_pC_s^2-8C_s^3) + 4C_s(C_p+C_s)(C_p-2C_s)^2 \times 3\\
&=16 C_p^2(C_p^2-2C_s^2) + 16C_pC_s^3 +48 C_s^4>0.
\end{aligned}
$$
Since $g(0)<0$ there exists by continuity a
$\hat{\delta}\in(0,\bar{\delta})$ s.t. $g(\hat{\delta})=0$ and we know
that $g$ is an increasing function.  This implies, because $f(0)=0$
that $f$ is a strictly increasing function
for $\delta<\hat{\delta}$, and a strictly decreasing
function for $\delta>\hat{\delta}$, and by a direct calculation, we
find for the second derivative
$$
\begin{aligned}
&\textstyle f''(\delta) = - \frac{4 \sqrt{2\omega^3(C_p^2-C_s^2)} (C_p+2C_s)^2 (C_p-C_s) \expo^{\frac{\omega\sqrt{C_p^2-C_s^2}}{C_sC_p}\delta}}{ (\sqrt{C_s}C_p)^3 \left(C_p\expo^{\frac{2\omega\sqrt{C_p^2-C_s^2}}{C_sC_p}\delta} + 2C_s\expo^{\frac{\omega\sqrt{C_p^2-C_s^2}}{C_sC_p}\delta} +C_p\right)^3} \\
&\textstyle\times \left[ \left(\expo^{\frac{4\omega\sqrt{C_p^2-C_s^2}}{C_sC_p}\delta}-1 \right) C_pC_s + 2(2C_p^2-C_s^2)\expo^{\frac{\omega\sqrt{C_p^2-C_s^2}}{C_sC_p}\delta} \left(\expo^{\frac{2\omega\sqrt{C_p^2-C_s^2}}{C_sC_p}\delta}-1 \right) \right]<0,
\end{aligned}
$$
therefore $\hat{\delta}$ is the absolute maximum for $f$.  Since
$\lim_{\delta\to\infty}f(\delta)=-\infty$, its graph will
cut the x-axis only once. By solving the equation $f(\delta)=0$
w.r.t. $\delta$ we find
$$
\begin{aligned}
{\delta}^*(C_p,C_s,\omega) &= \frac{C_s\sqrt{C_p^2-C_s^2}(C_p+2C_s)^2}{C_p\omega(C_s+C_p)} \frac{\expo^{2\alpha}-1}{C_p\expo^{2\alpha}+2\expo^{\alpha}C_s+C_p}\\
&= \dfrac{C_s\sqrt{C_p^2-C_s^2}(C_p+2C_s)^2}{C_p\omega(C_s+C_p)}\dfrac{\sinh(\alpha)}{C_p\cosh(\alpha) +C_s},
\end{aligned}
$$
where $\alpha$ the positive root of
$$
\begin{aligned}
0 &= \left[ (\alpha-1)\left(C_p^3-3C_p^2C_s+4C_s^3\right) \expo^{2\alpha}+2\alpha C_p^2C_s\expo^{\alpha}+(\alpha+1)\left(C_p^3+3C_p^2C_s-4C_s^3\right)\right]\\
&\Longleftrightarrow 
  \alpha C_p^2\left(C_p\cosh(\alpha)+C_s\right) =
  \left(C_p^3+\VD{(\alpha-1)(}3C_p^2C_s-4C_s^3\VD{)}\right)\sinh(\alpha).
\end{aligned}
$$
Note that $\alpha=0$ is also a solution but since ${\delta}^*>0$ we must have $\alpha>0$.
\end{proof}

\section{Numerical results}\label{sec:numerical}

\MG{We discretize} the Navier equations \MG{by} a finite element
method \MG{using a triangulation $\mathcal{T}_h$ of the computational
  domain $\Omega$, and obtain} a linear system $A \boldsymbol{u} =
\boldsymbol{b}$ \MG{to solve. To present} the \MG{discretized} Schwarz
method\MG{s, l}et $\{\mathcal{T}_{h,i}\}_{i=1}^N$ be a non-overlapping
partition of \MG{the triangulation $\mathcal{T}_h$, obtained} by using
a mesh partitioner like METIS~\cite{Karypis:98:METIS}. The overlapping
partition needed in \MG{the Schwarz methods} is defined as
follows\MG{: f}or an integer value $l \geq 0$, we build the
decomposition $\{\mathcal{T}_{h,i}^{l}\}_{i=1}^N$ such that
$\mathcal{T}_{h,i}^{l}$ is \MG{the} set of all triangles from
$\mathcal{T}_{h,i}^{l-1}$ and all triangles from $\mathcal{T}_h
\setminus \mathcal{T}_{h,i}^{l-1}$ that have non-empty intersection
with $\mathcal{T}_{h,i}^{l-1}$, and $\mathcal{T}_{h,i}^0 =
\mathcal{T}_{h,i}$.  With this definition the width of the overlap
\MG{is $2l$ mesh layers}. \MG{We denote by} $W_h$ the finite element
space associated with $\mathcal{T}_h$, \MG{and by} $W_{h,i}^{l}$ the
local finite element spaces on $\mathcal{T}_{h,i}^{l}$, \MG{which
  form} a triangulation of $\Omega_i$. Let $\mathcal{N}$ be the set of
indices of degrees of freedom of the global finite element space $W_h$
and $\mathcal{N}_i^{l}$ the set of indices of degrees of freedom of
the local finite element spaces $W_{h,i}^{l}$ for $l \geq 0$.  We
define the restriction operators from the global set of degrees of
freedom to the local one by $ {R}_i: W_h \rightarrow W_{h,i}^{l}$.  At
\MG{the} discrete level this is a rectangular matrix
$|\mathcal{N}_i^{l}| \times |\mathcal{N}|$ \MG{containing zeros and
  ones} such that if $\boldsymbol{v}$ is the vector of degrees of
freedom of $v_h \in W_h$, then ${R}_i \boldsymbol{v}$ is the vector of
degrees of freedom of $W_h$ in $\Omega_i$. The extension operator from
$W_{h,i}^{l}$ to $W_h$ and its associated matrix are then given by
$R_i^T$.  In addition we introduce a partition of unity $D_i$ as a
diagonal matrix $|\mathcal{N}_i^{l}| \times |\mathcal{N}_i^{l}|$ such
that
\begin{equation}
  \MG{I} = \sum_{i=1}^N R_i^T D_i R_i,
\end{equation}
where $\MG{I} \in \mathbb{R}^{|\mathcal{N}| \times |\mathcal{N}|}$ is
the identity matrix. With these ingredients we can now present the
\MG{Restricted Additive Schwarz} (RAS) preconditioner as described in
\cite[Chapter~1.4]{Dolean:15:DDM}\MG{,}
\begin{equation}\label{eq:dd_RAS}
  M^{-1}_{RAS} = \sum_{i=1}^N R_i^T D_i \left(R_i A R_i^T\right)^{-1} R_i.
\end{equation}
In our experiments we will also use the Optimized RAS (ORAS)
preconditioner which is based on local boundary value problem\MG{s}
with absorbing boundary conditions. In this case, let ${B}_i$ be the
matrix associated to a discretization of the corresponding local
\MG{problems} on the \MG{sub}domains $\Omega_i$ with \MG{absorbing}
boundary conditions on $\partial\Omega_i\cap \partial\Omega_j$. The
definition \MG{of the preconditioner is then} very similar to
\eqref{eq:dd_RAS} except that $R_i A R_i^T$ is replaced by
${B_i}$\MG{,}
\begin{equation}\label{eq:dd_ORAS}
  M^{-1}_{ORAS} = \sum_{i=1}^N R_i^T {D_i B_i}^{-1} R_i.
\end{equation}
It has been shown \MG{in \cite{gander2008schwarz} that the discretized
  parallel} Schwarz algorithm is equivalent to the stationary
iteration
\begin{equation}\label{eq:iterSchwarz}
\boldsymbol{u}^{n+1} = \boldsymbol{u}^n + M^{-1}\left(\boldsymbol{b} - A\boldsymbol{u}^n\right),
\end{equation}
where \MG{the preconditioner} $M^{-1}$ can \MG{either be}
$M^{-1}_{RAS}$ from \eqref{eq:dd_RAS} or $M^{-1}_{ORAS}$ \MG{from}
\eqref{eq:dd_ORAS}\MG{; see \cite{Stcyr:07:OMA} for the precise result
  for the latter which contains an algebraic condition. For more
  information on the influence of the partition of unity, see
  \cite{Gander:2019:DTP}.}

\MG{The stationary iteration \eqref{eq:iterSchwarz} can be accelerated
  using a Krylov method, which is equivalent to solving} the
preconditioned system
\begin{equation}
\label{eq:precSchwarz}
M^{-1}A \mathbf{u} = M^{-1}\mathbf{b}
\end{equation}
\MG{using the} Krylov method\MG{, see
  e.g. \cite[Chapter~3]{Dolean:15:DDM}. We test our new Schwarz
  methods both as stationary iterations and as preconditioners for a
  Krylov method.} \VD{In all the following test cases, we use as} stopping
criteri\MG{on} the relative $L^2$ norm of \MG{the} error,
$$
  \frac{\|\boldsymbol{u} -
  \boldsymbol{u}_n\|_{L^2(\Omega)}}{\|\boldsymbol{u} -
  \boldsymbol{u}_0\|_{L^2(\Omega)}} < 10^{-6},
$$ 
where $\boldsymbol{u}$ is the \MG{mono-}domain solution and
$\boldsymbol{u}_m$ denotes the approximation of $\boldsymbol{u}$ at
the $m$-th iteration of the iterative solver.
\VD{Note that }\MG{when using Krylov
  acceleration}, \VD{we can also use} \MG{the relative residual to stop the iteration,
  which is also available when the solution $\boldsymbol{u}$ is not
  known.}

We use a zero initial guess\footnote{\MG{When studying optimized
    parameters, starting with a zero initial guess is not advisable,
    see \cite[end of subsection 5.1]{gander2008schwarz}.}}  in all
\MG{our} tests, and we vary the size of the overlap and the type of
the decomposition (uniform or using METIS).  Numerical simulations
were done using the open source software Freefem++
\cite{Hecht:2012:ff++}, which is a high level language \MG{for the}
variational discretization of partial differential equations.

\subsection{Two-subdomain case: optimized Schwarz with TTC}

\MG{We first illustrate Theorem \ref{OverlappingSchwarzTTCTheorem}
  which states that the optimized} Schwarz algorithm with Taylor
transmission conditions can \MG{have converge problems for frequencies
  $k$ slightly bigger than $\frac{\omega}{C_s}$ if the overlap is not
  big enough. We} use the parameters $C_p=1$, $C_s=\frac{1}{2}$,
$\rho=1$, the domain $\Omega=(-1,1)\times(0,1)$ \MG{with Dirichlet
  conditions on top and bottom, and absorbing boundary conditions on
  the left and right, and the two subdomains
  $\Omega_1=(-1,\MG{\delta})\times(0,1)$ and
  $\Omega_2=(-\MG{\delta},1)\times(0,1)$. We discretize the
  time-harmonic Navier equations using uniform P1 finite elements
  with mesh size $h=\frac{1}{80}$.}
We show in Figure \ref{FigNumExp1}
\begin{figure}
  \centering
 \mbox{\includegraphics[width=0.55\textwidth,clip]{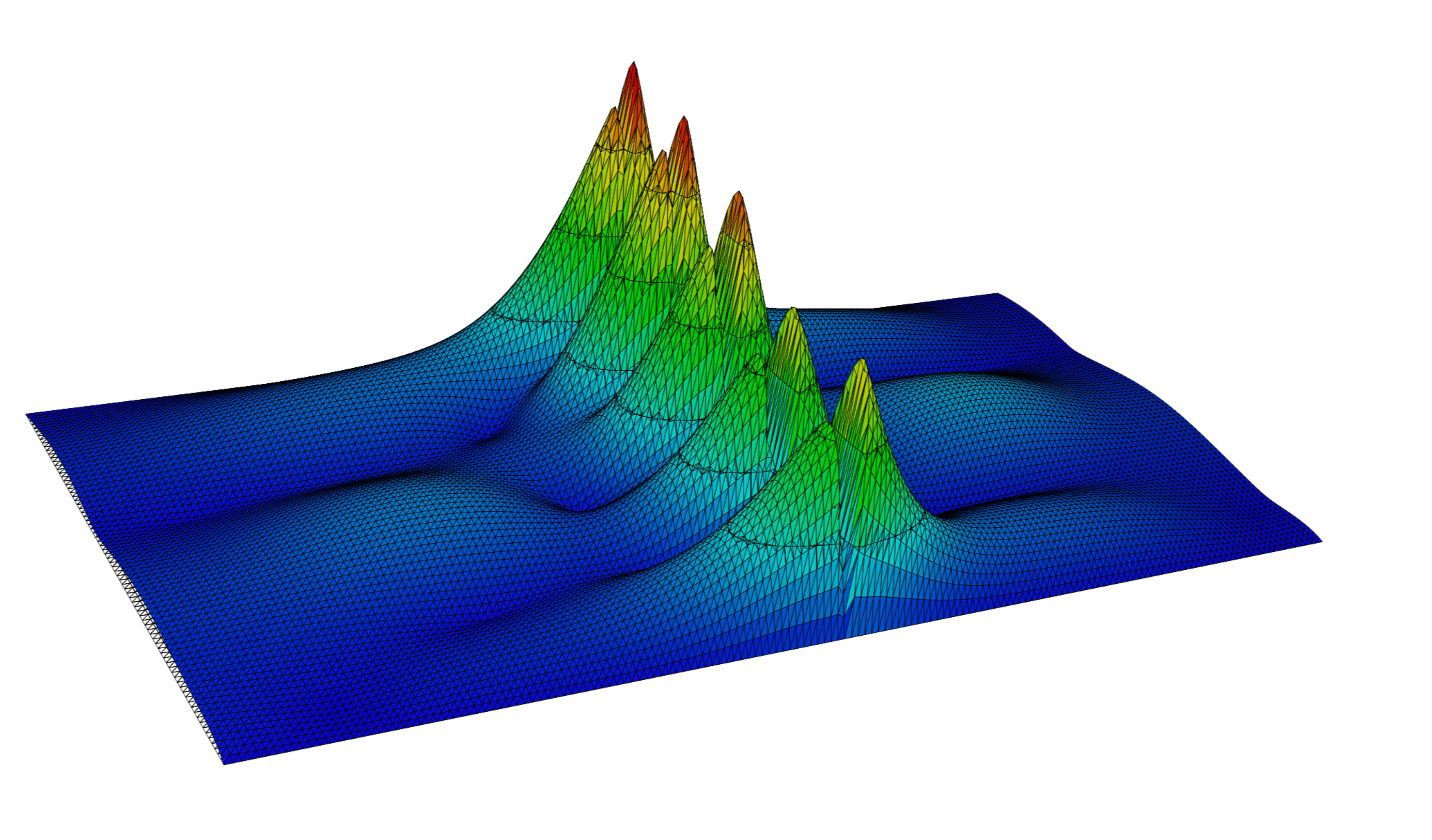}
  \includegraphics[width=0.45\textwidth,clip]{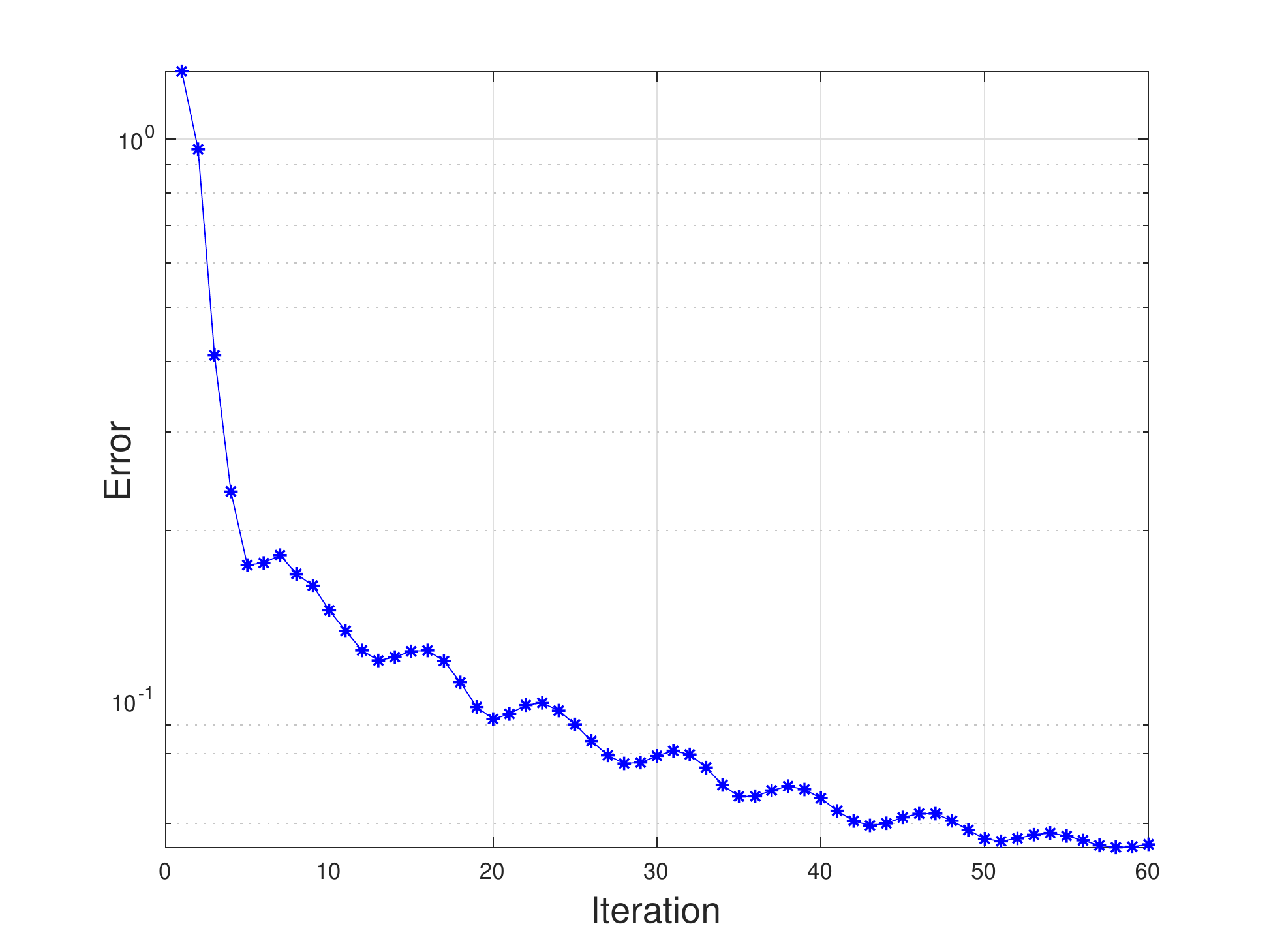}}
  \caption{\MG{Left:} Error in modulus at iteration 60 of the
    \MG{parallel optimized} Schwarz method with TTC \MG{for two
      subdomains}. \MG{Right: corresponding} convergence history
    ($\omega=5$, \MG{small overlap} $\delta=2h$).}
  \label{FigNumExp1}
\end{figure}
\MG{on the left} the error in modulus at iteration $60$ of the
optimized Schwarz method \MG{with} Taylor transmission conditions for
$\omega=5$ and overlap parameter $\delta = 2h$. We see that the
\MG{optimized Schwarz method stops converging: the interval for
  convergence problems predicted by Theorem
  \ref{OverlappingSchwarzTTCTheorem}} is
$\left[\frac{\omega}{C_s},k^{\star}\right] = [10,k^{\star}]$\MG{, and
  we observe} that the error \MG{on the left in Figure
  \ref{FigNumExp1}} has $5$ bumps along the interface which
corresponds well to the mode $|\sin(ky)|$ along the interface for
$k=5\pi\approx 15\MG{>\frac{\omega}{C_s}=10}$.

If we increase the overlap\MG{, $\delta=6h$, we see in Figure
  \ref{FigNumExp2} on the right}
 \begin{figure}
  \centering
 \mbox{\includegraphics[width=0.55\textwidth,trim=0 0 0 40,clip]{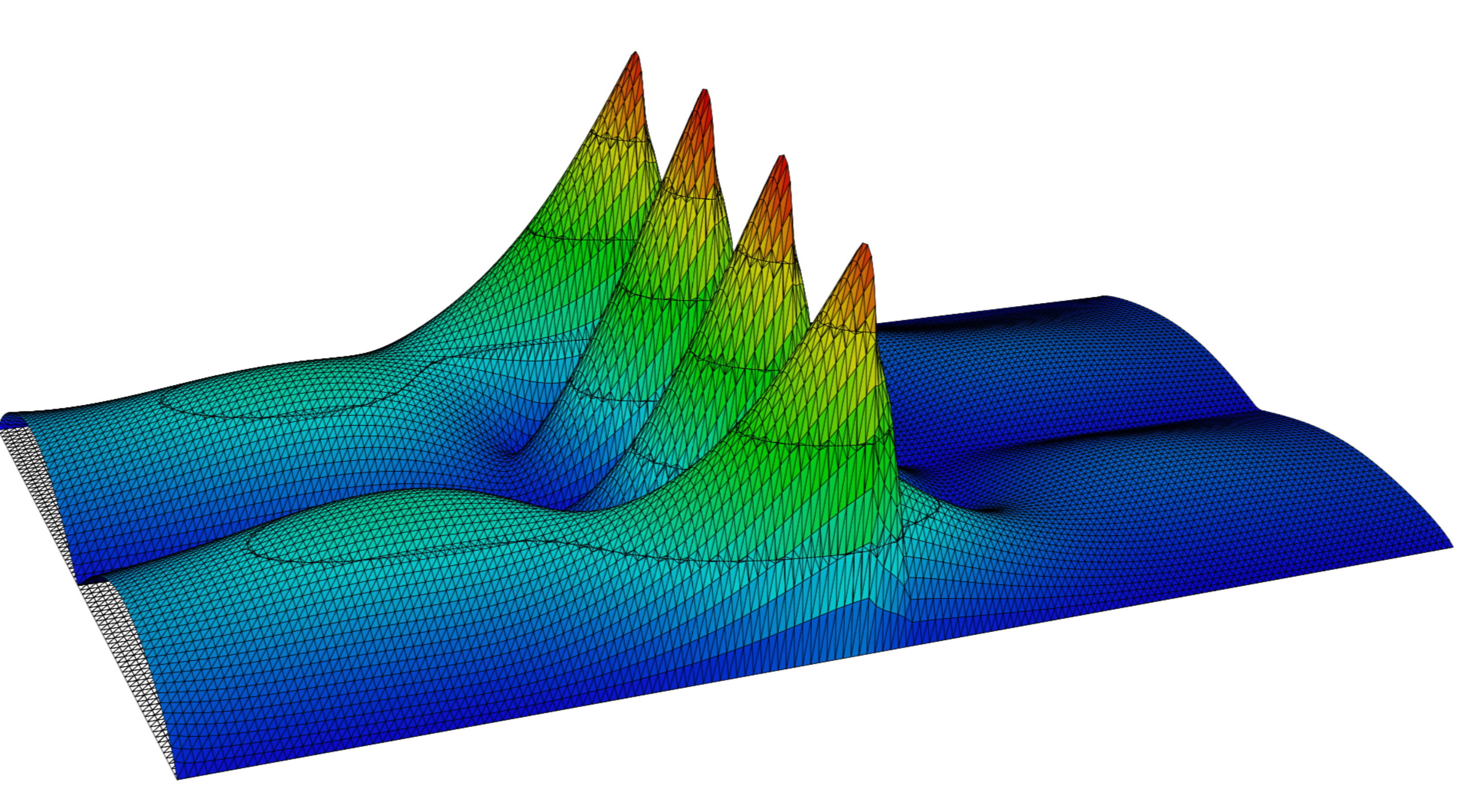}
  \includegraphics[width=0.45\textwidth,clip]{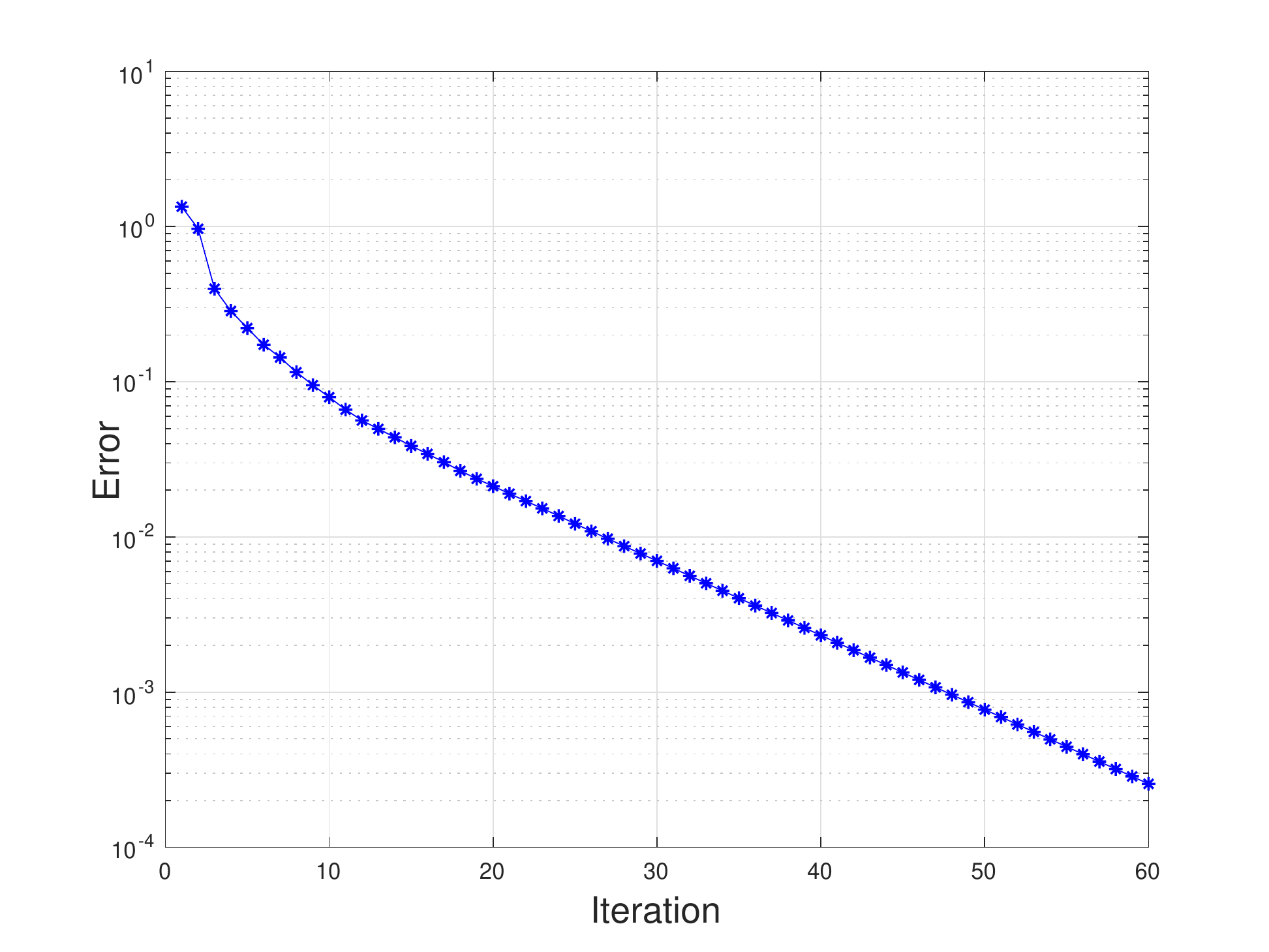}}
  \caption{\MG{Left: error} in modulus at iteration 60 of the
      \MG{optimized} Schwarz method with TTC and \MG{for two
        subdomains. Right: corresponding} convergence history
      ($\omega=5$, \MG{larger overlap} $\delta=6h$).}
  \label{FigNumExp2}
\end{figure}
 \MG{that the optimized Schwarz method is now converging.  The most
   slowly converging mode is shown on the left in Figure
   \ref{FigNumExp2}, and it also corresponds to a mode $|\sin(ky)|$
   along the interface with $k=4\pi\approx 12>\frac{\omega}{C_s}=10$,
   so our Fourier analysis captures accurately the convergence behavior
   of the optimized Schwarz method.}

\subsection{Comparing Schwarz as solver and preconditioner}

\MG{We next compare the performance of the Schwarz methods as solvers
  and preconditioners.} We simulate the wave propagation through a
computational domain given by the unit square \MG{$\Omega:=(0,1)^2$
  with absorbing} boundary conditions $\left({\cal T}^{({\bf n})} -
\rm i \mathbf{\sigma}_{{\bf n}}\right){\bf u} = \mathbf{g}$\MG{, where
  in} the two-dimensional case considered here
\begin{equation}
\mathbf{\sigma}_{{\bf n}} \MG{:}=\omega\rho \begin{pmatrix}
c_p n_x^2 + c_s n_y^2 & (c_p-c_s)n_x n_y \\
(c_p-c_s)n_x n_y & c_p n_y^2 + c_s n_x^2
\end{pmatrix}.
\end{equation}
\MG{The source term $\mathbf{g}$ is} chosen such that the exact
solution is a plane wave $\mathbf{u}^{inc}$ consisting \MG{of} both P-
and S-waves, $ \mathbf{u}^{inc}\MG{:}= {\bf d}\, e^{i\kappa_p {\bf
    x}\cdot {\bf d}}+ {\bf d}^{\perp}\, e^{i\kappa_s {\bf x}\cdot {\bf
    d}},\, {\bf d}= \left(\cos\left(\frac{\pi}{3}\right),
\cos\left(\frac{\pi}{3}\right)\right)^T$.  \MG{We choose the physical
  parameters $C_p=1$, $C_s=0.5$, $\rho=1$, $\lambda =
  \rho(C_p^2-2C_s^2)$, $\mu = \rho C_s^2$, and $\omega = 5$. We
  decompose the square domain $\Omega$} into $4\times 4$ \MG{equal}
subdomains \MG{$\Omega_i$} having each $40 \times 40$ discretization
points for a total number of 6400 \MG{degrees of freedom} per
subdomain.  The \MG{convergence of the Schwarz algorithms as solvers
  and preconditioners for GMRES} for different values \MG{of the}
overlap is shown in Figure \ref{FigNumExp3}.
\begin{figure}
  \centering
  \mbox{\includegraphics[width=0.5\textwidth,clip]{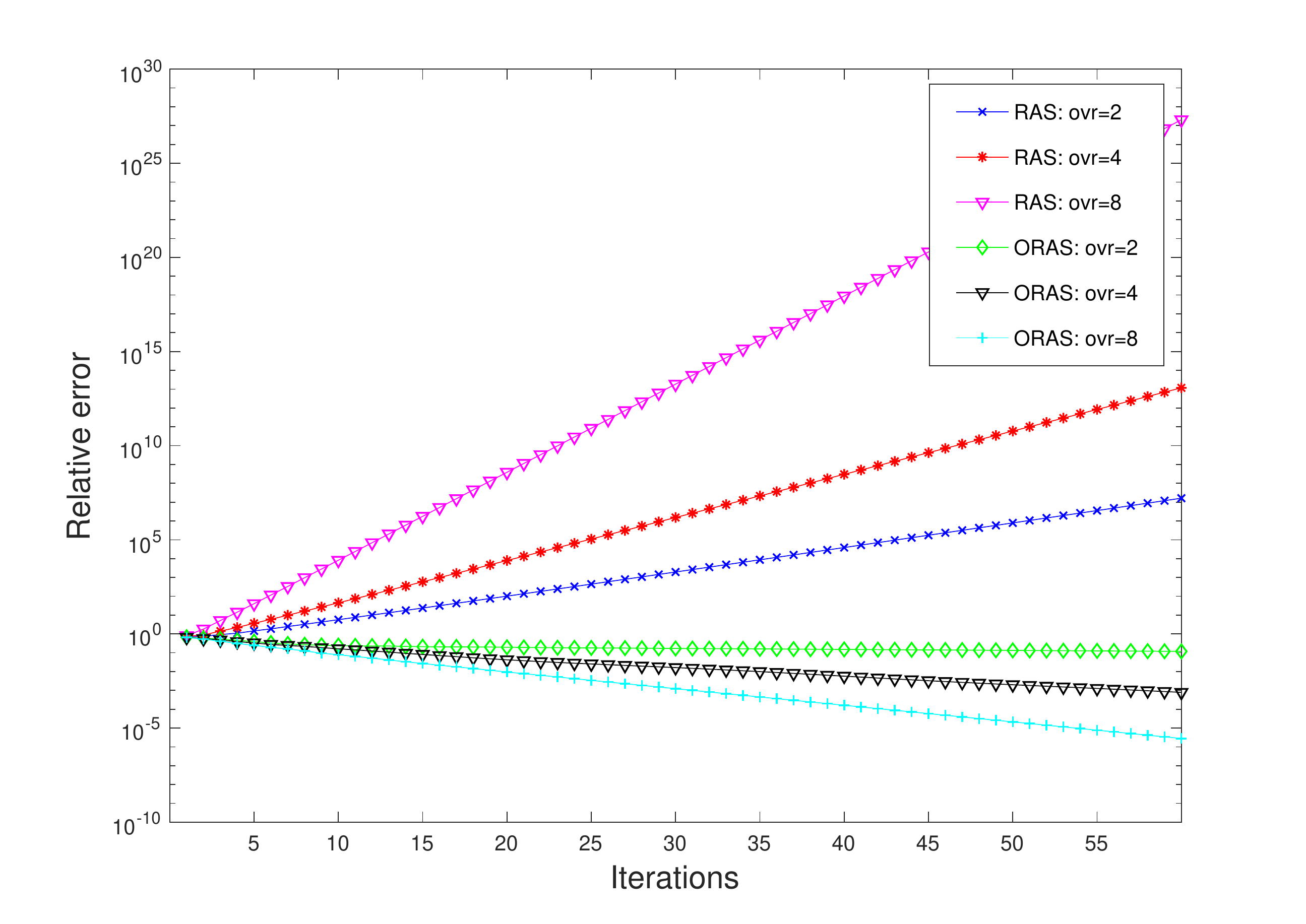}
    \includegraphics[width=0.5\textwidth,clip]{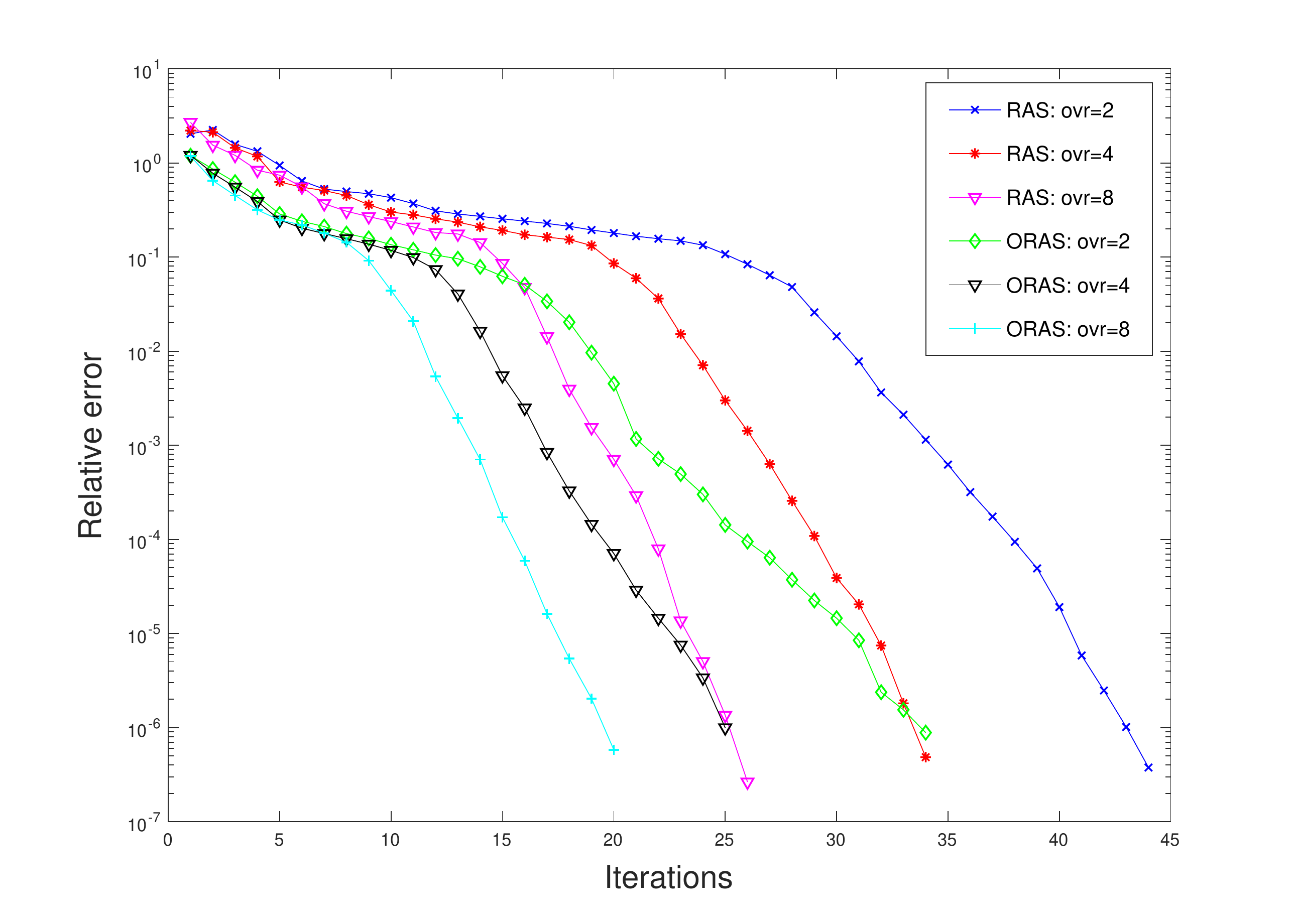}}
  \caption{Convergence history for RAS and ORAS as solvers (left) and
    preconditioners (right) for $\omega=5$, \MG{and different values of
    the overlap} $\delta$.}
  \label{FigNumExp3}
\end{figure}
As expected, the \MG{optimized Schwarz algorithm as solver converges,
  and the classical Schwarz algorithm diverges, for any size of the
  overlap. By increasing the overlap, as predicted by our two
  subdomain analyses in Theorem \ref{OverlappingSchwarzTTCTheorem} and
  \ref{ClassicalSchwarzTheorem}}, the \MG{optimized Schwarz} algorithm
is getting better, whereas \MG{classical Schwarz} is getting
worse. \MG{With GMRES acceleration, overlap also helps the classical
  Schwarz algorithm, but it still takes substantially more iterations
  to converge than the optimized one.}

\subsection{Solving a circular transmission problem}

\MG{We finally test our Schwarz methods for the Navier equations on a
  transmission problem formed by a circular inner part with radius 0.5
  that has different material characteristics from the surrounding
  outer part, truncated with absorbing boundary conditions at the
  radius 1. The heterogeneous physical parameters are given in Table
  \ref{tab:heter}.}
\begin{table}
\centering
  \tabcolsep0.5em
  \begin{tabular}{| c | c | c | c | c | c | c | c | c | c |} \hline
  Domain & E & $\nu$ & $\rho$ & $\mu$ & $\lambda$ & $C_p$ & $C_s$ & f & $\omega$  \\ \hline
 $r<0.5$ & $2.10^{11}$ & 0.3 & 7800 & $77.10^{9}$ & $12.10^{10}$ & 5927 & 3142 & $10^{4}$ & $2\pi 10^{4}$ \\ 
  $0.5\leq r\leq 1$ & $2.10^{11}$ & 0.47 & 7800 & $68.10^{9}$ & $11.10^{11}$ & 12588 & 2952 & $10^{4}$ & $2\pi 10^{4}$\\ \hline
  \end{tabular}
  \caption{Physical characteristics for the heterogeneous
    \MG{transmission problem.}}
\label{tab:heter}
\end{table}
We \MG{use METIS to partition} the unit disk $\{(x,y)| x^2+y^2\le 1\}$
into $4$ subdomains as \MG{shown} in Figure \ref{Partition} \MG{on the
left}.
\begin{figure}
  \centering
  \mbox{\includegraphics[width=0.5\textwidth,clip]{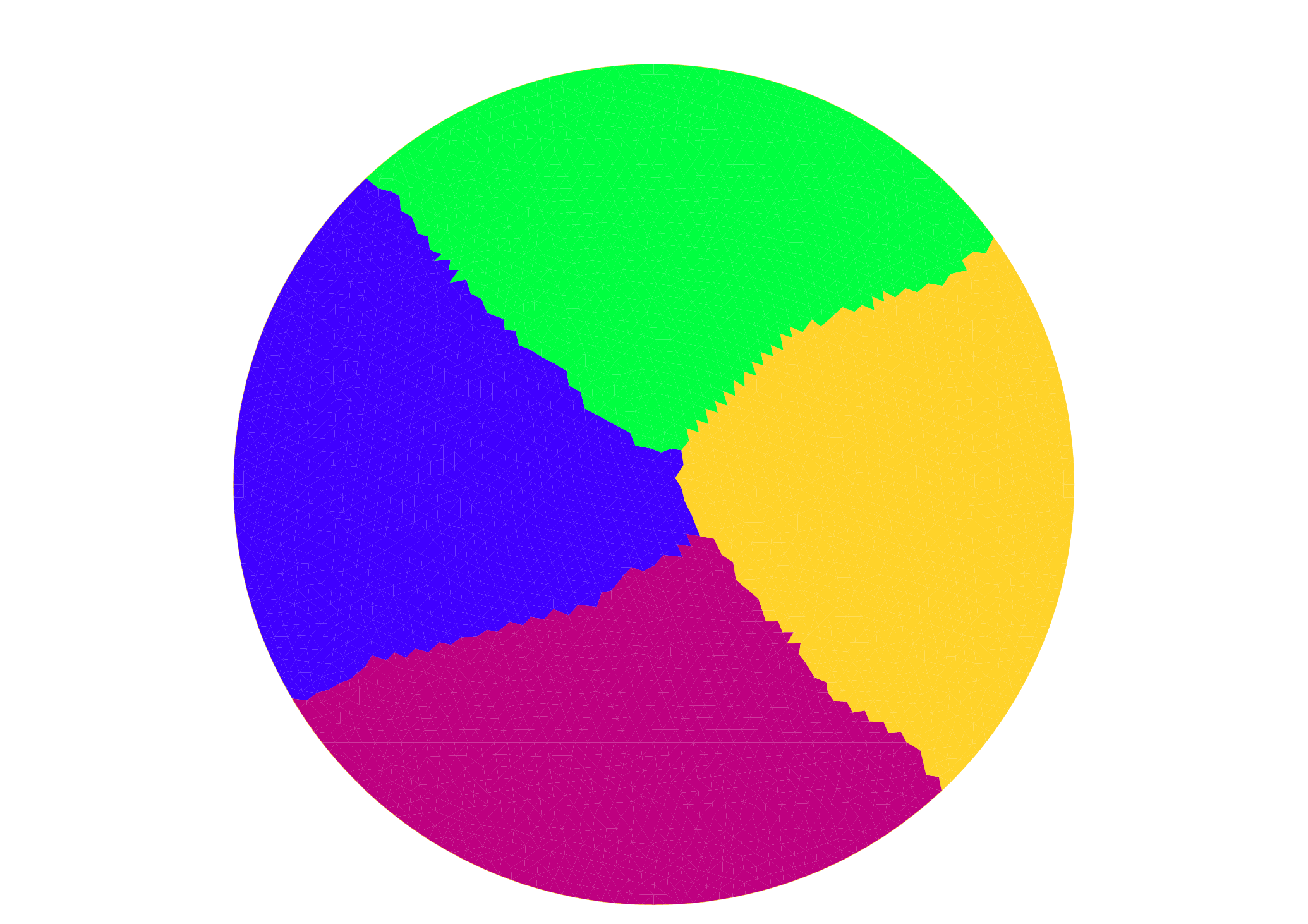}
  \includegraphics[width=0.5\textwidth,clip]{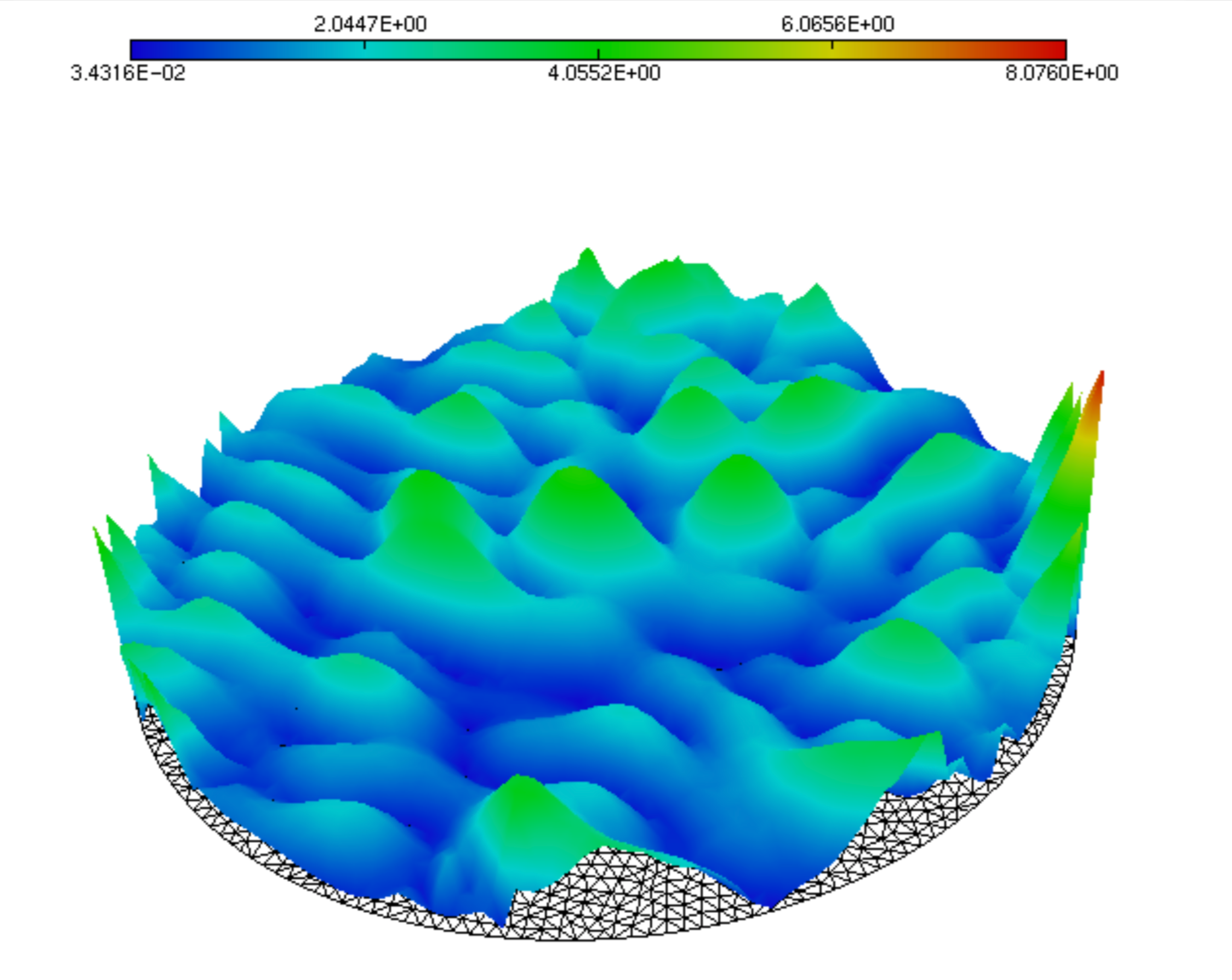}}
  \caption{METIS partition into 4 subdomains and \MG{monodomain} solution.}
  \label{Partition}
\end{figure}
\MG{The solution of the transmission problem we compute is shown in
  Figure \ref{Partition} on the right.  We test the different Schwarz
  methods again both as solvers and as preconditioners for GMRES; the
  corresponding results are shown} in Figure \ref{FigNumExp4}.
\begin{figure}
  \centering
  \mbox{\includegraphics[width=0.5\textwidth,clip]{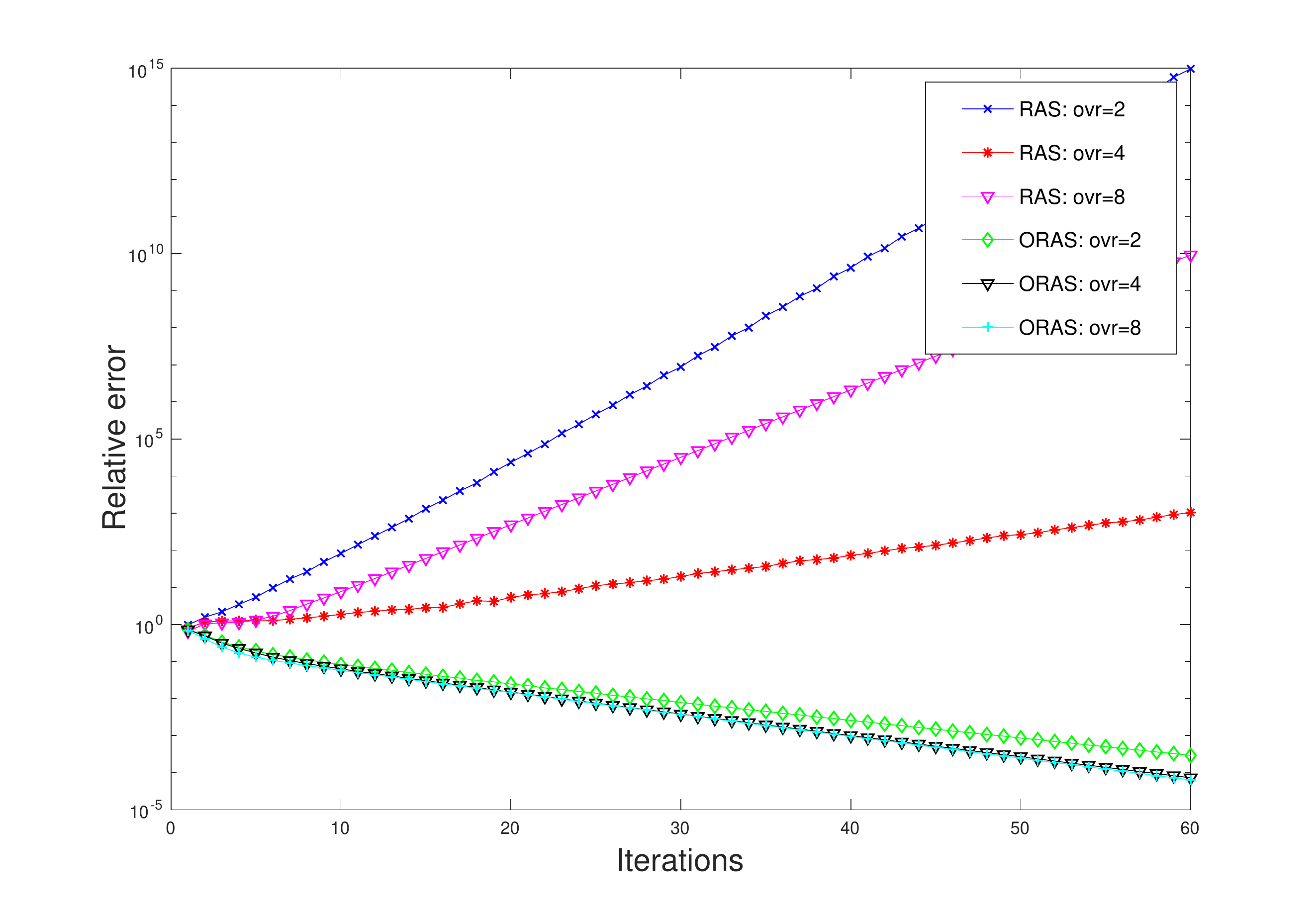}
    \includegraphics[width=0.5\textwidth,clip]{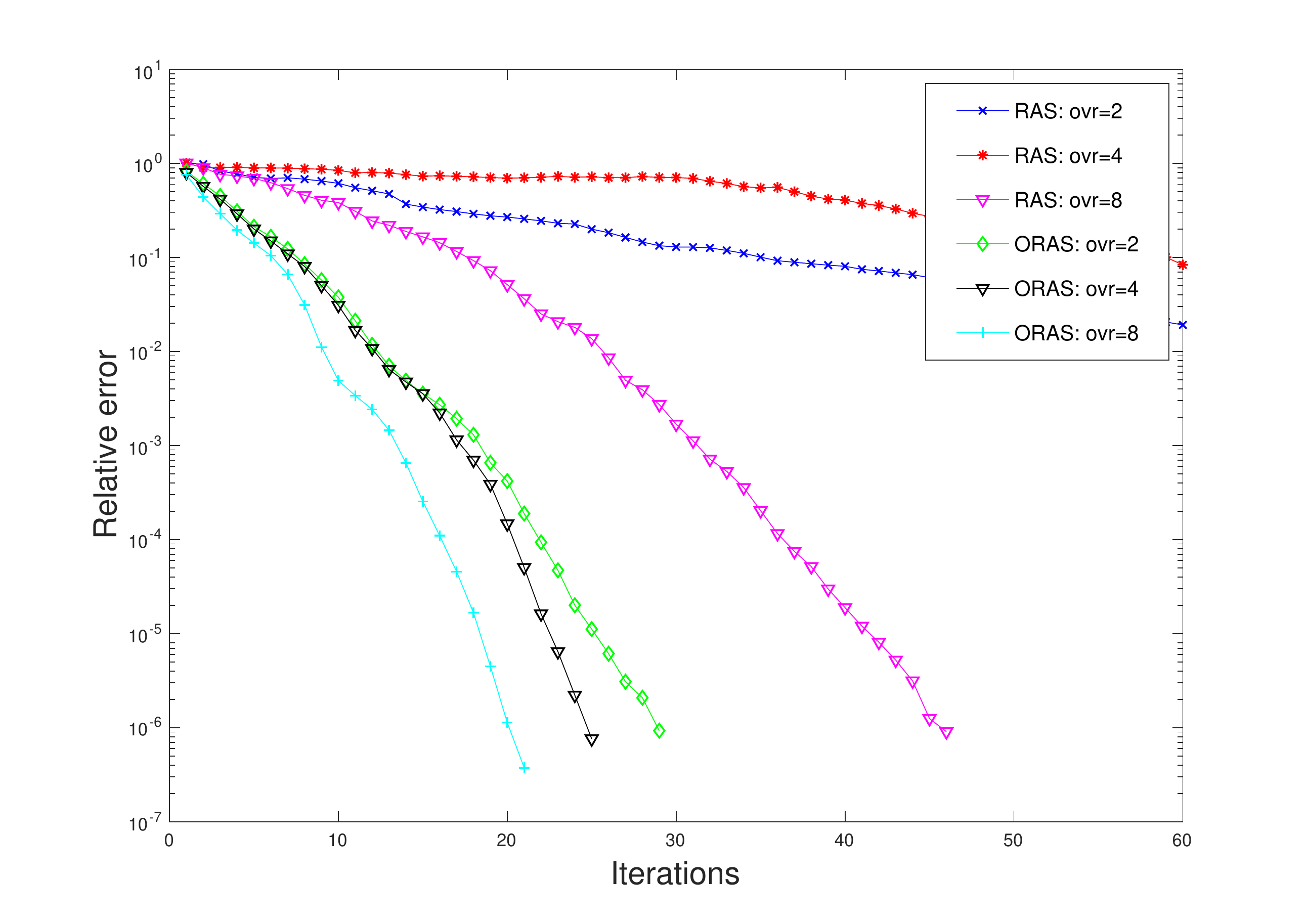}}
  \caption{Convergence history for \MG{classical and optimized Schwarz
      used} as solvers (left) and preconditioners (right) for the
    \MG{transmission} problem, and different values of
    \MG{the overlap} $\delta$.}
  \label{FigNumExp4}
\end{figure}
\MG{We see again that only the optimized Schwarz method with TTC
  converges when used as an iterative solver, the classical one
  diverges. This leads then naturally to a much better preconditioner
  for GMRES in the optimized Schwarz case for solving the
  transmission problem.}

\section{Conclusions}

\MG{We presented a first study of the applicability of Schwarz methods
  for the solution of time-harmonic elastic waves modeled by the
  Navier equations. We showed by a detailed and technical analysis for
  two subdomains that the classical Schwarz method can not converge
  when applied to the Navier equations. We then introduced more
  physical transmission conditions and showed that optimal
  transmission conditions exist which make the algorithm converge in
  two steps. Since these optimal transmission conditions involve
  non-local operators, we also introduced a local, low-frequency
  approximation, and proved that the new, optimized Schwarz method is
  then convergent, provided the overlap is large enough. We then
  tested the Schwarz methods both for the two subdomain case, and also
  for many subdomains, including a heterogeneous transmission problem,
  and we observed numerically that the new, optimized Schwarz method
  can indeed be used as an iterative solver, while the classical one
  can not, since it is divergent. The new transmission conditions lead
  also to a much better Schwarz preconditioner for GMRES than the classical
  ones. Our analysis opens the path to further development,
  namely transmission conditions which do not only improve the low
  frequency behavior, but improve the convergence over the entire
  spectrum of the iteration operator, a topic which we are currently
  investigating. }

\bibliographystyle{plain}
\bibliography{paper}

\begin{thebibliography}{10}

\bibitem{Alonso:06:NND}
Ana Alonso-Rodriguez and Luca Gerardo-Giorda.
\newblock New nonoverlapping domain decomposition methods for the harmonic
  {M}axwell system.
\newblock {\em SIAM J. Sci. Comput.}, 28(1):102--122, 2006.

\bibitem{BrunetPhD2018}
Romain Brunet.
\newblock {\em Domain Decomposition Methods for Time-Harmonic Elastic Waves}.
\newblock PhD thesis, University of Strathclyde, 2018.

\bibitem{Brunet:2019:CCS}
Romain Brunet, Victorita Dolean, and Martin~J. Gander.
\newblock Can classical {S}chwarz methods for time-harmonic elastic waves
  converge?
\newblock In {\em Domain Decomposition Methods in Science and Engineering XXV}.
  LNCSE, Springer, 2019.
\newblock submitted.

\bibitem{Cai:99:RAS}
X.-Ch. Cai and M.~Sarkis.
\newblock A restricted additive {S}chwarz preconditioner for general sparse
  linear systems.
\newblock {\em SIAM J. Sci. Comput.}, 21(2):792--797 (electronic), 1999.

\bibitem{Chen13a}
Z.~Chen and X.~Xiang.
\newblock A source transfer domain decomposition method for {H}elmholtz
  equations in unbounded domain.
\newblock {\em SIAM J. Numer. Anal.}, 51:2331--2356, 2013.

\bibitem{Chen13b}
Z.~Chen and X.~Xiang.
\newblock A source transfer domain decomposition method for {H}elmholtz
  equations in unbounded domain {P}art {II}: Extensions.
\newblock {\em Numer. Math. Theor. Meth. Appl.}, 6:538--555, 2013.

\bibitem{Chevalier:1998:MNT}
Philippe Chevalier.
\newblock {\em M\'ethodes num\'eriques pour les tubes hyperfr\'equences.
  R\'esolution par d\'ecomposition de domaine}.
\newblock PhD thesis, Universit\'e Paris VI, 1998.

\bibitem{Chevalier:1998:SMO}
Philippe Chevalier and Fr{\'e}d{\'e}ric Nataf.
\newblock Symmetrized method with optimized second-order conditions for the
  {H}elmholtz equation.
\newblock In {\em Domain decomposition methods, 10 (Boulder, CO, 1997)}, pages
  400--407. Amer. Math. Soc., Providence, RI, 1998.

\bibitem{Collino:1997:NIM}
Francis Collino, G.~Delbue, Patrick Joly, and A.~Piacentini.
\newblock A new interface condition in the non-overlapping domain decomposition
  for the {M}axwell equations {H}elmholtz equation and related optimal control.
\newblock {\em Comput. Methods Appl. Mech. Engrg}, 148:195--207, 1997.

\bibitem{Despres:1990:DDP}
Bruno Despr{\'e}s.
\newblock D\'ecomposition de domaine et probl\`eme de {H}elmholtz.
\newblock {\em C.R. Acad. Sci. Paris}, 1(6):313--316, 1990.

\bibitem{Despres:1991:MDD}
Bruno Despr{\'e}s.
\newblock {\em M{\'e}thodes de d{\'e}composition de domaine pour les
  probl{\`e}mes de propagation d'ondes en r{\'e}gimes harmoniques}.
\newblock PhD thesis, Paris IX, 1991.

\bibitem{Dolean:15:DDM}
V.~Dolean, P.~Jolivet, and F.~Nataf.
\newblock {\em An introduction to domain decomposition methods}.
\newblock Society for Industrial and Applied Mathematics (SIAM), Philadelphia,
  PA, 2015.
\newblock Algorithms, theory, and parallel implementation.

\bibitem{Dolean:09:OSM}
Victorita Dolean, Luca~Gerardo Giorda, and Martin~J. Gander.
\newblock Optimized {S}chwarz methods for {M}axwell equations.
\newblock {\em SIAM J. Scient. Comp.}, 31(3):2193--2213, 2009.

\bibitem{Dolean:08:DDM}
Victorita Dolean, Stephane Lanteri, and Ronan Perrussel.
\newblock A domain decomposition method for solving the three-dimensional
  time-harmonic {M}axwell equations discretized by discontinuous {G}alerkin
  methods.
\newblock {\em J. Comput. Phys.}, 227(3):2044--2072, 2008.

\bibitem{ElBouajaji:12:OSM}
Mohamed {El Bouajaji}, Victorita Dolean, Martin~J. Gander, and Stephane
  Lanteri.
\newblock Optimized {S}chwarz methods for the time-harmonic {M}axwell equations
  with dampimg.
\newblock {\em SIAM J. Scient. Comp.}, 34(4):2048--2071, 2012.

\bibitem{EY1}
B.~Engquist and L.~Ying.
\newblock Sweeping preconditioner for the {H}elmholtz equation: {H}ierarchical
  matrix representation.
\newblock {\em Comm. Pure Appl. Math.}, LXIV:0697--0735, 2011.

\bibitem{EY2}
B.~Engquist and L.~Ying.
\newblock Sweeping preconditioner for the {H}elmholtz equation: {M}oving
  perfectly matched layers.
\newblock {\em Multiscale Model. Sim.}, 9:686--710, 2011.

\bibitem{Ernst:12:NAM}
Oliver~G. Ernst and Martin~J. Gander.
\newblock Why it is difficult to solve {H}elmholtz problems with classical
  iterative methods.
\newblock In {\em Numerical analysis of multiscale problems}, pages 325--363.
  Springer, 2012.

\bibitem{GanderAILU05}
M.~J. Gander and F.~Nataf.
\newblock An incomplete {LU} preconditioner for problems in acoustics.
\newblock {\em J. Comput. Acoust.}, 13:455--476, 2005.

\bibitem{gander2006optimized}
Martin~J. Gander.
\newblock Optimized {S}chwarz methods.
\newblock {\em SIAM Journal on Numerical Analysis}, 44(2):699--731, 2006.

\bibitem{gander2008schwarz}
Martin~J. Gander.
\newblock Schwarz methods over the course of time.
\newblock {\em Electron. Trans. Numer. Anal}, 31(5):228--255, 2008.

\bibitem{Gander:2019:DTP}
Martin~J. Gander.
\newblock Does the partition of unity influence the convergence of {S}chwarz
  methods?
\newblock In {\em Domain Decomposition Methods in Science and Engineering XXV}.
  LNCSE, Springer, 2019.
\newblock submitted.

\bibitem{gander2007optimized}
Martin~J Gander, Laurence Halpern, and Fr{\'e}d{\'e}ric Magoules.
\newblock An optimized {S}chwarz method with two-sided {R}obin transmission
  conditions for the {H}elmholtz equation.
\newblock {\em International journal for numerical methods in fluids},
  55(2):163--175, 2007.

\bibitem{Gander:2001:OSH}
Martin~J. Gander, Fr\'ed\'eric Magoul\`es, and Fr\'ed\'eric Nataf.
\newblock Optimized {S}chwarz methods without overlap for the {H}elmholtz
  equation.
\newblock {\em SIAM J. Sci. Comput.}, 24(1):38--60, 2002.

\bibitem{gander2016optimized}
Martin~J. Gander and Hui Zhang.
\newblock Optimized {S}chwarz methods with overlap for the {H}elmholtz
  equation.
\newblock {\em SIAM Journal on Scientific Computing}, 38(5):A3195--A3219, 2016.

\bibitem{ganderzhang2018SIREV}
Martin~J. Gander and Hui Zhang.
\newblock A class of iterative solvers for the {H}elmholtz equation:
  Factorizations, sweeping preconditioners, source transfer, single layer
  potentials, polarized traces, and optimized {S}chwarz methods.
\newblock {\em SIAM Review}, 61(1):3--76, 2019.

\bibitem{Hecht:2012:ff++}
F.~Hecht.
\newblock New development in freefem++.
\newblock {\em J. Numer. Math.}, 20(3-4):251--265, 2012.

\bibitem{Karypis:98:METIS}
G.~Karypis and V.~Kumar.
\newblock A software package for partitioning unstructured graphs, partitioning
  meshes, and computing fill-reducing orderings of sparse matrices.
\newblock Technical report, University of Minnesota, Department of Computer
  Science and Engineering, Army HPC Research Center, Minneapolis, MN, 1998.

\bibitem{Peng:10:ODD}
Zhen Peng, Vineet Rawat, and Jin-Fa Lee.
\newblock One way domain decomposition method with second order transmission
  conditions for solving electromagnetic wave problems.
\newblock {\em J. Comput. Phys.}, 229(4):1181--1197, 2010.

\bibitem{Saad:1986:GGM}
Youssef Saad and Martin~H. Schultz.
\newblock {GMRES}: A generalized minimal residual algorithm for solving
  nonsymmetric linear systems.
\newblock {\em SIAM J. Sci. Stat. Comp.}, 7:856--869, 1986.

\bibitem{Stcyr:07:OMA}
A.~St-Cyr, M.~J. Gander, and S.~J. Thomas.
\newblock Optimized multiplicative, additive, and restricted additive {S}chwarz
  preconditioning.
\newblock {\em SIAM J. Sci. Comput.}, 29(6):2402--2425 (electronic), 2007.

\bibitem{Stolk}
C.~C. Stolk.
\newblock A rapidly converging domain decomposition method for the {H}elmholtz
  equation.
\newblock {\em J. Comput. Phys.}, 241:240--252, 2013.

\bibitem{Huttunen:04:UWV}
Huttunen Tomi, Monk Peter, Collino Francis, and Kaipio~Jari P.
\newblock The ultra-weak variational formulation for elastic wave problems.
\newblock {\em SIAM Journal on Scientific Computing}, 25(5):1717--1742, 2004.

\bibitem{Toselli:2005:DDM}
Andrea Toselli and Olof Widlund.
\newblock {\em Domain Decomposition Methods - Algorithms and Theory}, volume~34
  of {\em Springer Series in Computational Mathematics}.
\newblock Springer, 2005.

\bibitem{ZD}
L.~Zepeda-N{\'u}{\~n}ez and L.~Demanet.
\newblock The method of polarized traces for the 2{D} {H}elmholtz equation.
\newblock {\em J. Comput. Phys.}, 308:347--388, 2016.

\end{thebibliography}

\end{document}